\documentclass[]{article}
\usepackage[affil-it]{authblk}

\usepackage{fullpage}
\usepackage{graphicx}
\usepackage{pdflscape}
\usepackage{secdot}
\usepackage{subcaption}
\usepackage{amsmath, amssymb,amsthm}
\usepackage{amsfonts}
\usepackage{latexsym}
\usepackage{booktabs}
\usepackage{enumitem,blindtext}
\usepackage[numbers]{natbib} %
\usepackage[export]{adjustbox}
\usepackage{mathrsfs}
\usepackage[usenames,dvipsnames]{color}
\usepackage[mathcal]{eucal}
\usepackage{caption}
\usepackage[linesnumbered, noend]{algorithm2e}
\usepackage{setspace}
\usepackage{tikz}
\usepackage{tikzsymbols}
\usetikzlibrary{arrows}
\usetikzlibrary{calc}

\newtheorem{definition}{Definition}
\newtheorem{example}{Example}
\newtheorem{theorem}{Theorem}
\newcommand{\comments}[1]{}
\DeclareGraphicsExtensions{.png,.pdf,.jpg}

\newcommand{\blue}[1]{{\color{black}#1}} 
 
\definecolor{purple}{RGB}{153,50,204}

\newcommand{\SLP}{$S$-LP}

\newcommand{\first}{(F1)}
\newcommand{\second}{(F2)}

\usepackage[normalem]{ulem}

\usepackage{hyperref}
\begin{document}

	\title{Algorithmic expedients for the $S$-labeling problem\footnote{\textcopyright 2019. This manuscript version is made available under the CC-BY-NC-ND 4.0 license \url{http://creativecommons.org/licenses/by-nc-nd/4.0/}. Accepted for publication in Computers \& Operations Research; doi: 10.1016/j.cor.2019.04.014}}
	
	\author[1]{Markus Sinnl\thanks{markus.sinnl@univie.ac.at}}
	
	\affil[1]{Department of Statistics and Operations Research, Faculty of Business, Economics and Statistics, University of Vienna, Vienna, Austria \newline Institute of 
		Production and Logistics Management, Johannes Kepler University Linz, Linz, Austria}

	\date{}
	
	\maketitle

\begin{abstract}
Graph labeling problems have been widely studied in the last decades and have a vast area of application. 
In this work, we study the recently introduced $S$-labeling problem, in which the nodes get labeled using labels from 1 to $|V|$ and for each edge the contribution to the objective function, called $S$-labeling number of the graph, is the minimum label of its end-nodes. The goal is to find a labeling with minimum value. 
The problem is NP-hard for planar subcubic graphs, although for many other graph classes the complexity status is still unknown. 

In this paper, we present different algorithmic approaches for tackling this problem: We develop an exact solution framework based on Mixed-Integer Programming (MIP)
which is enhanced with valid inequalities, starting and primal heuristics and specialized branching rules. We show that our MIP formulation has no integrality gap for paths, cycles and perfect $n$-ary trees, and,
to the best of our knowledge, we give the first polynomial-time algorithm for the problem on $n$-ary trees as well as a closed formula for the $S$-labeling number for such trees. 
Moreover, we also present a Lagrangian heuristic and a constraint programming approach. 
A computational study is carried out in order to 
(i) investigate if there may be other special graph classes, where our MIP formulation has no integrality gap, 
and (ii) assess the effectiveness of the proposed solution approaches for solving the problem on a dataset consisting of general graphs.
\end{abstract}

\section{Introduction and motivation}

Graph labeling problems have been widely studied in the last decades and they have a vast area of application,~e.g.,
coding theory~\cite{bloom1977applications}, computational biology~\cite{karp1993mapping}, 
computer networks~\cite{jin2005graph}, design of error-correcting codes~\cite{rodriguez2008effective} radio channel assignment~\cite{van1998graph} and more (we refer the reader to the surveys~\cite{diaz2002survey,dehghan2013algorithmic,gallian2009dynamic} for further details).
In such problems, we are typically interested in assigning distinct positive integers (i.e., \emph{labels}) to the nodes and/or edges of the graph subject to some constraints, such that a given objective function is optimized. 
In this work, we concentrate on the \emph{$S$-labeling problem}, which is defined as follows.
\begin{definition}[The $S$-labeling problem (\SLP)]
Let $G=(V,E)$ be a graph and $\phi:V \rightarrow \{1, \ldots, |V|\}$ be a labeling of the nodes. The $S$-labeling number $SL_\phi(G)$ with respect to the labeling $\phi(G)$ is defined as $\sum_{\{u,v\} \in E} \min\{\phi(u),\phi(v)\}$. The goal of the $S$-labeling problem (\SLP) is to find a labeling $\phi^*(G)$, such that $SL_{\phi^*}(G)$ has minimum value amongst all possible labelings for $G$.

\end{definition}

In the remainder of this paper, we simply write $\phi^*$ and $SL_{\phi^*}$, when the graph is clear from the context. 
The \SLP\ was introduced in~\cite{vialette2006packing} in the context of packing $(0,1)$-matrices, where it was shown to be \emph{NP-complete} for planar subcubic graphs.
It is studied in detail in a series of papers~\cite{fertin2009s,fertin2015algorithmic,fertin2017s}, which focused on deriving properties of optimal labelings. 
Based on these properties, exact and approximation algorithms for some graph classes are developed: Polynomial-time exact algorithms are given for \emph{caterpillar graphs} (trees, where all nodes are at most distance one from the central path) and \emph{split graphs} (graphs which can be partitioned into a clique and an independent set). 
Moreover, a greedy approximation algorithm, which gives a labeling $\phi$ with $SL_\phi< \frac{|E|(|V|+1)}{3}$ (resp., $\frac{|E|(|V|+1)}{4}$ if the graph is acyclic with the maximum degree of a node at least three) is presented. 
This algorithm is shown to be a $\frac{4|E|\Delta}{3|V|}$ approximation algorithm for general graphs, where $\Delta$ is the maximum degree of a node in the graph, and a $4/3$ approximation algorithm for regular graphs; for some other special graphs, the approximation factor is further refined. 
Additionally, a fixed-point parameter tractable algorithm is also presented;
this approach is based on partial enumeration, where the parameters is a positive integer $k$ and the question is whether there is a labeling with $SL_\phi<|E|+k$. 
Finally, closed formulas for $SL_{\phi^*}$ for complete graphs, paths and cycles are given, but without proof. 
We note that one result of~\cite{fertin2015algorithmic} (Lemma 3) claims that in any optimal labeling $\phi^*$, the node $i$ with label one is a node with maximum degree in the graph. 
We have a counter-example to this result, see Section~\ref{sec:analyze}.

\paragraph{Contribution and paper outline} In this paper, we develop different modeling and algorithmic approaches to characterize and solve the \SLP. 
The paper is organized as follows. In Section~\ref{sec:form}, \blue{we present two Mixed-Integer Programming (MIP) formulations (denoted as \first\ and \second) for the problem.} We describe an exact solution framework based on these MIPs, which includes starting and primal heuristics to construct high-quality feasible solutions during the solution process as well as a specialized branching-scheme. \blue{Moreover, we show that formulation \first\ can be obtained as projection of formulation \second.}
In Section~\ref{sec:analyze} we investigate the Linear Programming (LP) relaxation of the MIPs presented in Section~\ref{sec:form}, and show that for paths, cycles and perfect $n$-ary trees the LP-relaxation exhibits no integrality gap. 
This is done by presenting a combinatorial procedure for solving the dual of the LP-relaxation and ad-hoc polynomial-time primal algorithms. To the best of our knowledge, we present the first polynomial-time algorithm for solving the \SLP\ on perfect $n$-ary trees. We also give a closed formula to compute $SL_{\phi^*}$ for perfect $n$-ary trees. 

In Section~\ref{sec:alternative}, we present further solution methods, namely a Lagrangian heuristic and a constraint programming approach.
%
Section~\ref{sec:res} contains a computational study. The purpose of this study is twofold: (i) to computationally investigate, if there are additional classes of graphs, where the LP-relaxation of our MIP formulation may have no integrality gap; and (ii) to assess the performance of the presented approaches for solving the \SLP\ on general graphs.
Finally, Section~\ref{sec:con} concludes the paper.

\section{Mixed-Integer Programming approaches for the \SLP}
\label{sec:form}

In this section, we present MIP (exact) approaches for solving the \SLP.
We first provide two MIP formulations for the problem and show that one of them is the projection of the other. After this, we describe branch-and-cut schemes based on these MIPs.

%
%
%

For the proposed formulations, let $x^k_i$, with $i,k\in\{1,\ldots,|V|\}$, be a binary variable such that $x^k_i=1$, if node $i$ (or, more formally, the node with index $i$) 
gets labeled with the number $k$, and $x^k_i=0$, otherwise. 
The following constraints~\eqref{eq:a-sum} and~\eqref{eq:b-sum} ensure that each label gets used exactly once, and that each node gets one label,
\begin{align} 
\sum_{i \in\{1,\ldots,|V|\} } x^k_i&=1 &\quad \forall k \in\{1,\ldots,|V|\} \label{eq:a-sum} \tag{UNIQUE}\\
\sum_{k \in\{1,\ldots,|V|\} } x^k_i&=1 &\quad \forall i \in\{1,\ldots,|V|\}. \label{eq:b-sum} \tag{ONELABEL} 
\end{align} 

To define the first formulation, for each edge $e \in E$ we introduce continuous variables $\theta_e\geq 0$, which are used to measure the contribution of edge $e$ to the objective. 
Using these variables, our first formulation \first\ for the \SLP\ is given by
\begin{align} 
\first \quad SL_{\phi^*} = \min\quad &\sum_{e \in E} \theta_e  & \label{eq:obj1} \tag{F1.1} \\
&\eqref{eq:a-sum}, \eqref{eq:b-sum} & \notag \\
\theta_e & \geq k - \sum_{l<k} (k-l) (x^l_{i}+x^l_{i'}) \label{eq:opt} & \quad \forall k\in\{1,\ldots,|V|\}|,\; \forall e=\{i,i'\} \in E \tag{F1.2} \\
 x^k_{i} & \in\{0,1\}&\; \forall i,k\in\{1,\ldots,|V|\}. \notag 
\end{align}
Constraints~\eqref{eq:opt} ensure that for each edge $e$ the correct contribution is counted in the objective function: For an edge $e=\{i,i'\}$, let $k^1$ and $k^2$ (assume wlog $k^1<k^2$) be the labels of the two end-nodes $i$ and $i'$ for a given solution encoded by $\bar{\mathbf{x}}$. 
For each $k\leq k^1$, the right-hand-side (rhs) of~\eqref{eq:opt} is $k$, for $k^1 < k < k^2$, the rhs is $k^1$, and for $k \geq k^2$, the rhs is $k^1-k^2<k^1$. 
Thus, $\theta_e$ takes the value $k^1$.

For the second model \second, instead of using continuous variables $\theta_e\geq 0$ to measure the contribution of edge $e$, we introduce binary variables $d^k_e$, with $k\in\{1,\ldots, |V|\}$, such that $d^k_e = 1$ iff the contribution of edge $e$ to the objective is $k$, $d^k_e = 0$ otherwise \blue{(as the contribution of an edge is the minimum of the label of its two end-nodes, it is enough to have these variables with $k$ up to $|V|-1$, but for ease of readability we include all of them in the formulation)}. With these variables, we can formulate the problem as follows,
\begin{align} 
\second \quad SL_{\phi^*} = \min\quad &\sum_{e \in E} \sum_{k\in\{1,\ldots,|V|\}} k d^k_e & \label{eq:obj2} \tag{F2.1} \\
&\eqref{eq:a-sum}, \eqref{eq:b-sum} & \notag \\
\sum_{k\in\{1,\ldots,|V|\}} d^k_e&=1 &\quad \forall e \in E \label{eq:e-sum} \tag{F2.2} \\
d^k_e & \leq x^k_i+x^k_{i'}  &\quad k\in\{1,\ldots,|V|\}, \forall e=\{i,i'\} \in E \label{eq:link} \tag{F2.3}\\
 x^k_{i} & \in\{0,1\}&\;i,k\in\{1,\ldots,|V|\}
 \notag \\
  d^k_{e} & \in\{0,1\}&\;\forall e \in E,\;  k\in\{1,\ldots,|V|\}.
   \notag
\end{align}
Constraints~\eqref{eq:e-sum} ensure that exactly one of the variables $d^k_e$ is taken for each edge $e$, and constraints \eqref{eq:link} make sure that one of the two end nodes of edge $e$ has label $k$, if the associated  $d^k_e$-variable is chosen as one (and thus the contribution of the edge in the objective is $k$).

We observe that $\first$ has $O(|V|^2)$ variables, and $O(|V||E|)$ constraints, while $\second$  has $O(|V|^2+ |V| |E|)$ variables, and $O(|V||E|)$ constraints. Thus $\second$ has considerably more variables, on the other hand, constraints \eqref{eq:link} are much sparser than their counterpart~\eqref{eq:opt}. Note that the integrality of the variables $d_e^k$ can be relaxed, as for fixed variables $x \in\{0,1\}^{|V|\times |V|}$ the remaining coefficient matrix in $d_e^k$ becomes totally unimodular.

Finally, the following family of inequalities is valid for \second;
\begin{theorem}
Suppose the three nodes $i,i',i'' \in V$ form a triangle $e=\{i,i'\},e'=\{i,i''\},e''=\{i',i''\}$ in $G$ and let $V' \subset \{1,\ldots,|V|\}$ be a subset of labels. Then
\begin{equation}
\sum_{k \in V'} \big(d^k_{e}+d^k_{e'}+d^k_{e''}\big)\leq 1+ \sum_{k \in V'} \big( x^k_{i}+x^k_{i'}+x^k_{i''} \big) \label{eq:valid} \tag{TRIANGLE}
\end{equation}
is valid for \second.
\end{theorem}

\begin{proof}
Clearly, the sum on the left-hand-side (lhs) can be at most three. 
If the sum is zero or one, validity is trivial and validity for two follows easily due to constraints~\eqref{eq:link}. 
Thus, suppose the sum on the lhs is three, this means all three edges of the triangle are in some of the label-levels given in $V'$. 
However, in this case, at least two of the variables on the right-hand-side must be one, as clearly not all three edges can be on the same labeling-level in any feasible solution.
\end{proof}

As there is an exponential number of~\eqref{eq:valid} inequalities, 
we do not add them in the beginning, but separate them on-the-fly, i.e., we embed their separation in \blue{a} branch-and-cut scheme; the separation is described in Section~\ref{sec:further}.
The computational results in Section~\ref{sec:resgen} reveal that these inequalities can be quite helpful in some instances. 

\subsection{Comparison between $\first$ and $\second$\label{sec:strength}}

In this section, we show that formulation $\first$ can be obtained from $\second$ by using Benders decomposition to project out the $d^k_e$-variables using Benders optimality cuts, i.e., consider a Benders master problem consisting of~\eqref{eq:a-sum},~\eqref{eq:b-sum} and variables $\theta_e$ that account for the contribution of edge $e$ to the objective function, the obtained optimality cuts are exactly inequalities~\eqref{eq:opt}.


In order to show this equivalence, we relax the integrality of the $d^k_e$-variables, 
and consider the dual of $\second$ for a fixed solution encoded by vector $\bar x \in \{0,1\}^{|V|\times |V|}$.
In the following, for ease of readability, when we refer to dual of some formulation, 
we mean the dual of the corresponding LP-relaxation (with possible some subset of the variables fixed as mentioned).
Observe that for fixed solution given by $\bar x$, 
the problem decomposes into one problem for each edge $e$.
Let $\gamma_e$ be the dual-variables associated with constraint~\eqref{eq:e-sum}, 
and let $\delta^k_e$ be the dual variables associated with constraints~\eqref{eq:link}. 
The dual $(D_{\second(\bar x,i,i')})$ of $\second$, for a given $\bar x$ and an edge $e=\{i,i'\}$, then reads as
\begin{align} 
(D_{\second(\bar x,i,i')}) \quad \max\quad  \gamma_e -\sum_{k\in \{1,\ldots, |V|\} }& (\bar x^k_i+\bar x^k_{i'}) \delta^k_ e  & \label{eq:objdual1} \tag{$D_{\second(\bar x,i,i')}.1$} \\
\gamma_e-\delta_e^k &\leq k &\quad  \forall k\in \{1,\ldots, |V|\}\label{eq:dual1-2} \tag{$D_{\second(\bar x,i,i')}.2$}  \\
\delta_e^k &\geq 0 &\quad \forall k\in \{1,\ldots, |V|\} \tag{$D_{\second(\bar x,i,i')}.3$} 
\end{align}

Note that using standard duality rules, we would obtain $\delta_e^k \leq 0$ and positive coefficients associated with these variables in the objective and constraints.
For ease of exposition, we write these dual variables as $\delta_e^k \geq 0$ by changing the objective function coefficients accordingly.

The optimal solution of $(D_{\second(\bar x,i,i')})$ can be derived as follows.
It is easy to see that in order to maximize the objective function~\eqref{eq:objdual1}, $\gamma_e$ needs to be increased.
However, due to constraints~\eqref{eq:dual1-2}, if $\gamma_e$ is increased to $k$,
to ensure feasibility we need to compensate this, by setting $\delta_e^l = k-l$,
whenever $l<k$. Due to constraints~\eqref{eq:a-sum} and~\eqref{eq:b-sum}, 
the coefficients $(\bar x^k_i+\bar x^k_{i'})$, of dual variable $\delta^k_e$,
take value of one only for two values of $k$, say $l^1, l^2$ with $l^1 < l^2$, and value of zero for
the remaining values of $k$. 
Thus, as long as $k$ is at most $k^1=l^1+1$, setting $\delta_e^l = k-l$ does not influence the objective function, 
since the coefficients of the corresponding $\delta^l_e$ variables are zero,
and the objective value obtained is $k^1$. 
For $k\in [k^1,\ldots,l^2]$, the objective function also has the value $k^1$:
If the value of $\gamma_e$ is increased by one within this interval, then the value of $\delta^{k^1}_e$,
whose objective function coefficient is minus one must also be increased by one.
Finally, starting from $k^2=l^2+1$, also $\delta^{k^2}_e$ would need to be increased by one to ensure feasibility, i.e., for each increase of $\gamma_e$ by one, 
two variables with coefficient minus one in the objective must also be increased by one. 
Thus, the optimal solution value of the dual for the given $\bar x$ is $k^1$,
and it is achieved by applying above procedure for any $k$ with $k^1\leq k <k^2$. The values of the dual variables in an optimal solution are $\gamma_e=k$ and $\delta^l_e=k-l$ for $l<k$ and zero otherwise. 
Thus, the associated Benders optimality cuts are inequalities~\eqref{eq:opt}.

\subsection{Starting and primal heuristic \label{sec:heur}}

In order to construct a feasible starting solution to initialize the branch-and-\blue{cut framework} , 
we use the greedy algorithm proposed in~\cite{fertin2017s}, which works as follows.
We pick any unlabeled node, say $i$, with maximum degree (ties broken randomly),
and we label it with the smallest available label;
then, we remove $i$ from the graph and repeat the procedure until all nodes are labeled. 
After this, we try to improve the obtained solution with a local search-phase which is based on exchanging the label of pairs of nodes. 
The complete scheme is described in Algorithm~\ref{alg:primal}. 
In the algorithm, $\phi^{-1}_H(k)$ denotes the node with label $k$. 
In the local search-phase (lines~\ref{algline:lsalpha}-\ref{algline:lsomega}), 
we exploit the fact that an exchange of labels between pairs of nodes $i,i'$ does not require
to recalculate the labeling number from scratch, as it is enough to recalculate the contribution of the edges $e\in E$ such that $e:\{i,\cdot\}$ or $e:\{i',\cdot\}$. We also do not try all label-pair-exchanges, but for a given label $k$ limit the change to labels $k'\leq min(k,maxContribLabel)$, where $maxContribLabel$ is the maximum contribution to the objective of any edge, where one end-node is labeled with $k$.

\begin{algorithm}[h!tb]   
\DontPrintSemicolon                 
\SetKwInOut{Input}{input}\SetKwInOut{Output}{output}
\Input{instance $G=(V,E)$ of the \SLP}
\Output{feasible labeling $\phi_H$ with labeling number $z^H=SL_{\phi_H}$}
$G'\gets G$\;
$z^H \gets 0$\;
\tcc{greedy algorithm phase}
\For{$1\leq k \leq |V|$}
{
$i \gets $ node with maximal degree in $G'$\;
$\phi_H(i) \gets k$\;
$G' \gets G' \setminus i$\;
\For{$\{i,i'\} \in E$ with $i'$ unlabeled}
{
$z^H \gets z^H+k$\;
}
}
\tcc{local-search phase}
$improve \gets true$\;
\Repeat{$improve=false$}
{\label{algline:lsalpha}
$improve \gets false$\;
\For{$1\leq k \leq |V|$}
{
$i \gets \phi^{-1}_H(k)$ \;
$cost^i \gets 0$\;
$maxContribLabel \gets 0$ \;
\For{ $\{i,i''\}\in E$ }
{
$cost^i \gets cost^i+\min(\phi_H(i),\phi_H(i''))$\;
$maxContribLabel \gets \max(maxContribLabel,\min(\phi_H(i),\phi_H(i'')) )$\;
}

\For{$1 \leq k' \leq \min(k,maxContribLabel)$}
{
$i' \gets \phi^{-1}_H(k')$ \;
$cost^{i'} \gets 0$\;
\For{ $\{i',i''\}\in E$ }
{
$cost^{i'} \gets cost^{i'}+\min(\phi_H(i'),\phi_H(i''))$\;
}
$costExchange \gets 0$\;
\For{ $\{i,i''\}\in E$ }
{
$costExchange \gets costExchange+\min(k',\phi_H(i''))$\;
}
\For{ $\{i',i''\}\in E$ }
{
$costExchange \gets costExchange+\min(k,\phi_H(i''))$\;
}
\If{$costExchange < cost^{i} + cost^{i'}$}
{
$z^H \gets z^H-cost^{i} - cost^{i'} + costExchange$\;
$\phi_H(i) \gets k'$\;
$\phi_H(i') \gets k$\;
$improve \gets true$\;
\textbf{break}\;
}
}

}
}\label{algline:lsomega}

\caption{Starting heuristic\label{alg:primal} }
\end{algorithm}

Based on this algorithm, we also designed a primal heuristic,
guided by the LP-relaxation, which is embedded within 
the branch-and-\blue{cut} search framework.
Let $\tilde x$ be the value of the $x$-variables of the LP-relaxation at a branch-and-\blue{cut} node. 
We solve the minimum assignment problem induced by constraints~\eqref{eq:a-sum} and~\eqref{eq:b-sum}, setting the coefficient of variable $x^k_i$ in the objective to $\tilde x^k_i$, the obtained solution gives a feasible labeling.
Afterwards, we attempt to improve the obtained solution by applying the local search-part of the heuristic described in Algorithm~\ref{alg:primal}.

\subsection{Further enhancements \label{sec:further}}

We now describe the additional ingredients that are part of the
developed branch-and-cut algorithm, namely a branching strategy, an initialization procedure and separation procedures for violated inequalities.

\paragraph{Branching}
Constraints~\eqref{eq:b-sum} (and also~\eqref{eq:a-sum} and~\eqref{eq:e-sum}) are \emph{generalized-upper-bound} (GUB) constraints (also known as \emph{special-ordered-sets}). 
It is well-known that in presence of such constraints, branching on a single variable $x_i^k$ 
may not be very efficient, as it will often lead to an ``unbalanced'' search tree:
in one part of the search tree, one variable $x_i^k$ is fixed to zero and in the other part, all other variables $x_i^{k'}, k'\neq k$ are fixed to zero. 
Instead, branching on a subset of the variables appearing in such a constraint could be more efficient, this strategy is often called GUB-branching 
(see, e.g.,~\cite{conforti2014integer,nemhauser1988integer,wolsey1998integer} or also~\cite{cho2015row} for more details on implementing such a branching-scheme). 
In our case, we implemented GUB-branching based on constraints~\eqref{eq:b-sum}, 
i.e., at every branch-and-\blue{cut} node, we select a node $i$ to branch on, and then make two children-nodes where we enforce
\begin{equation}
\sum_{k \in K} x^k_i=0 \quad \text{or} \quad \sum_{k \in  \{1,\ldots,|V|\} \setminus K} x^k_i=0,
\end{equation}
where $K \subset \{1,\ldots,|V|\}$. 
This is a valid branching scheme, as in one child, node $i$ must take a label in $ \{1,\ldots,|V|\} \setminus K$, and in the other it must take a label in $K$. 
Let $\tilde x$ be the value of the LP-relaxation at a branch-and-\blue{cut} node and suppose we already selected a node $i$. 
The partition $K$ is found by calculating $B=\lfloor \sum_{k \in  \{1,\ldots,|V|\}} k \tilde x^k_i \rfloor$ and putting every $k\leq B$ in $K$. 
To select node $i$ to branch on, for each node $i$, with fractional variables $\tilde x^k_i$, we calculate the following score $sc_i$ based on up-\emph{pseudocosts} $\psi^+(i,k)$ and down-\emph{pseudocosts} $\psi^-(i,k)$: $sc_i=\sum_{k \in  \{1,\ldots,|V|\}: \tilde x^k_i>0}\psi^+(i,k) (1- \tilde x^k_i)+\psi^-(i,k) \tilde x^k_i $, 
and select the node with the maximum score.
If the maximum score is under 0.001, we then proceed with the default branching of the branch-and-cut solver. \blue{In preliminary tests during implementation, we also briefly tried other scores, i.e., only using the $(1- \tilde x^k_i)$ or $\tilde x^k_i$ part, but did not see any improvements compared to using the sum of both parts.}	
Pseudocosts are a concept often used for branching decisions and are provided by CPLEX, 
which is the branch-and-cut solver we use (see, e.g.,~\cite{achterberg2005branching,conforti2014integer} for more on pseudocosts).

\paragraph{Initialization}
While both $\first$ and $\second$ have polynomially many constraints, for some instances, 
the size of the resulting models may still become prohibitive for efficient solving. 
Removing constraints initially, and only adding them on the fly when needed (i.e., using branch-and-cut) could be \blue{a} useful option in such a case. 
However, in preliminary computations, for formulation $\second$, such a strategy did not turn out to be computationally successful. 
This may be explained in the structure of the constraints.
Suppose that for some edge $e$, only a subset of the constraints~\eqref{eq:link} is added. Then the $d_e^k$-variable with the smallest $k$, 
for which no constraint was added will be take value of one in the resulting relaxation.
Thus, to ensure that the correct variable $d_e^{k'}$ is set to one, 
all constraints~\eqref{eq:link} with $k\leq k'$ must be present. 
On the contrary, for formulation $\first$, a branch-and-cut approach was more useful.
In this case, adding a single constraint~\eqref{eq:opt}, for each edge $e$, 
is enough for having an effect on the objective value of the resulting relaxation. 
In our initialization approach for \first, the set of initial constraints consists of all constraints~\eqref{eq:b-sum} and~\eqref{eq:a-sum}, and a single-constraint~\eqref{eq:opt} for each $e$. 
The latter are determined by heuristically solving the dual problem of $\second$ with the algorithm described in the next section. 
For each edge $e$, we add the constraint~\eqref{eq:opt} for the largest $k$, for which the dual multiplier for $d^k_e\leq x^k_i+x^k_{i'}$ is non-zero.
Separation of constraints~\eqref{eq:opt} is done by enumeration. 
During each separation round, we add one violated inequality (if there is any) for each $e$;
if there is more than one violated, we add the one with the largest violation. 
By checking the potential violation of~\eqref{eq:opt} for each $e$ increasingly by $k$, 
we do not need to re-calculate it from scratch for each $k$.

%

\paragraph{Separation of inequalities \eqref{eq:valid}}

We do not separate the inequalities for all subsets of labels $V'\subset \{1,\ldots,|V|\}$, 
but simply check for all increasing subsets, 
starting with $\{1,2\}$ until $\{1,2, \ldots |V|-1\}$. 
The separation is then done by enumeration, i.e., we enumerate all \emph{triangles} in the beginning and store them;
afterwards, during a separation round at a given node of the branch-and-\blue{cut} tree, 
we check the inequalities for each of the above mentioned subsets of $V'$. 
By performing this in an increasing order, we do not need to re-separate from scratch,
but just need to add the contribution from the current label. 
For each triangle, we stop either when a violated inequality is found, 
or when the sum of fractional variables on the right-hand-side of the inequality has reached the value of three (since it cannot grow larger than that).

\section{Analyzing the MIP-formulation for special classes of graphs \label{sec:analyze}}

Clearly, any feasible solution to the dual of an LP-relaxation of a MIP gives a lower bound, 
which can be used in, e.g., a branch-and-bound algorithm to solve the MIP. 
Depending on the structure and size of the resulting dual, 
a high-quality dual solution can potentially be found using a combinatorial procedure instead of solving the corresponding LP model directly using, e.g., the simplex-algorithm. 
For example, for the facility location problem~\cite{erlenkotter1978dual}, 
the Steiner tree problem~\cite{wong1984dual}, 
or the network design problem~\cite{balakrishnan1989dual},
there exist so-called \emph{dual ascent} algorithms, which (heuristically) solve the dual by starting with a solution, where all variables are zero (such a solution is feasible for these problems), 
and then systematically increase the dual variables in order to get a good dual solution. 
In this section we present a combinatorial algorithm for solving $(D_{\second})$, 
which is inspired by these approaches. 
We show that for paths, cycles and perfect $n$-ary trees, 
the dual solution produced by our heuristic algorithm is in fact optimal. 


\subsection{A heuristic algorithm for solving the dual of formulation \second \label{sec:dh}}



In addition to the dual multipliers defined in the previous section, 
let $\alpha_k$ be the dual multipliers associated with~\eqref{eq:a-sum},
and let $\beta_i$ be the dual multipliers associated with~\eqref{eq:b-sum}. 
The dual of \second\, denoted as $(D_{\second})$, 
is as follows (again changing the coefficients of $\delta_e^k$ to get $\delta_e^k \geq 0$);
\begin{align} 
(D_{\second}) \quad \max\quad  \sum_{k\in \{1,\ldots,|V| \}} \alpha_k+\sum_{i \in V} \beta_i+\sum_{e \in E}\gamma_e & \label{eq:objdual2} \tag{$D_{\second}.1$} \\
\alpha_k+\beta_i+\sum_{e:\{i,\cdot\} \in E} \delta_e^k &\leq 0 &\quad k\in \{1,\ldots,|V|\}, \forall i \in V \label{eq:dual2-1} \tag{$D_{\second}.2$}  \\
\gamma_e-\delta_e^k &\leq k &\quad \forall e \in E, k\in \{1,\ldots,|V|\}\label{eq:dual2-2} \tag{$D_{\second}.3$}  \\
\delta_e^k &\geq 0 &\quad \forall e \in E, k\in \{1,\ldots,|V|\} \tag{$D_{\second}.4$} 
\end{align}


Algorithm~\ref{alg:dual} outlines a (greedy) heuristic to solve $(D_{\second})$. We denote the degree of a node $i$ with $\delta(i)$. 
While Algorithm~\ref{alg:dual} is a heuristic for general graphs, it gives the optimal dual solution for paths, 
cycles and perfect $n$-ary trees due to the particular structure of these graphs. 
At the end of the section, we discuss an extended version of the algorithm, 
which can produce better dual bounds for other graphs and also give a small example of this. 
Note that in both variants of our algorithm the dual variables are always kept integral, 
so in case the optimal dual solution has fractional values, it cannot be achieved with our algorithms.

\begin{algorithm}[h!tb]   
	\DontPrintSemicolon                 
	\SetKwInOut{Input}{input}\SetKwInOut{Output}{output}
	\Input{instance $G=(V,E)$ of the \SLP}
	\Output{feasible solution $(\alpha,\beta,\gamma,\delta)$ of $(D_{\second})$ with value $z^D$}
	$(\boldsymbol{\alpha,\beta,\delta}) \gets 0, (\boldsymbol{\gamma}) \gets 1, z^D \gets |E|$ \;
	$\Delta=\max_{i \in V} \delta(i)$\;
	\For{$\bar k\in \{1,\ldots,|V|\} $}
	{
		\eIf{$|E|-\bar k \cdot \Delta>0$}
		{
			$z^D \gets z^D+|E|-\Delta$ \;
			\For{$k \leq \bar k$}
			{
				$\alpha_k \gets \alpha_k+\Delta$\;
				\For{$e \in E$}
				{
					$\delta^k_e \gets \delta^k_e+1$\;
				}
			}
			\For{$e \in E$}
			{
				$\gamma_e \gets \gamma_e+1$\;
			}
		}
		{
			\textbf{break}
		}
	}
	\caption{Heuristic solution procedure for the dual $(D_{\second})$\label{alg:dual} }
\end{algorithm}

We start with $(\boldsymbol{\alpha,\beta,\delta})$ set to zero, and all $\gamma_e$ set to one, which is a feasible solution, and we iteratively try to increase the values of some $\gamma_e$ (which appear with coefficient one in the objective function \eqref{eq:objdual2}) by one. Due to constraints \eqref{eq:dual2-1} and \eqref{eq:dual2-2}, increasing some $\gamma_e$ by one implies the increase of some $\delta^k_e$ and also the decrease of $\alpha_k$ or $\beta_i$ to preserve feasibility of the solution. In our algorithm, we only decrease $\alpha_k$ and keep $\beta_i$ at zero. As $\alpha_k$ appears with a coefficient of minus one in the objective function, our goal is to iteratively set the variables in such a way, that the net-change in the objective is positive at each step, and we stop, when the change is non-positive.

The key observations used in the design of the algorithm are the following: i) when we want to increase some $\gamma_e$ from $\bar k$ to $\bar k+1$, all $\delta^k_e$ with $1 \leq k \leq \bar k+1$ need to be increased by one to keep feasibility, which in turn implies decreasing by one all $\alpha_k$ with $1 \leq k \leq \bar k+1$; ii) when $\alpha_k$ is set to some value $\bar \alpha_k$, for all adjacent edges $e$ of any node $i$ we can set $\bar \alpha_k$ variables $\delta^k_e$ to one and the solution remains feasible.

In our algorithm, we start with $\bar k=1$ and a step consists of setting $\alpha_k$, $\delta^k_e$ and $\gamma_e$ for all $k \leq \bar k$ and some or all $e \in E$, then we proceed to the next step with $\bar k=\bar k+1$. In the simplified version, we set $\alpha_k$ always to the maximum degree $\Delta$ of a node in the graph. Thus, we can always increase all $\gamma_e$ at a step (as all $\delta^k_e$ can be increased), i.e., the positive change in the objective due to $\gamma_e$ is $|E|$ at every step. Moreover, the negative change in the objective at every step due to setting the $\alpha_k$ is $\Delta \cdot \bar k$, as all $\alpha_k$ up to $\bar k$ need to be increased by $\Delta$ to keep the feasibility. Thus, the algorithm stops, when $\bar k \cdot \Delta \geq |E|$, and at each step before, the dual objective gets increased by $|E|-\bar k  \cdot \Delta > 0$. Observe that the increase gets smaller at every step, as $\Delta \cdot \bar k$ grows with every step. The runtime of Algorithm \ref{alg:dual} is $O(|V|^2|E|)$

In the extended version of the algorithm, we do not just increase by $\Delta$ at a step, but try all the values up to $\Delta$. If $\alpha_k$ is set to a value $\bar \alpha_k$ lower than $\Delta$, for some $i \in V$, not all $\delta^k_e$ with $i \in e$ can be increased, which in turn means not all $\gamma_e$ can be increased. Thus, a subset of $e$ to increase needs to be chosen in these cases. We do this in a heuristic fashion: A list \texttt{active} of \emph{active edges} is kept, and only such edges are considered for the remaining steps (i.e., increasing $\gamma_e$, $\delta^k_e$ and calculation of $\Delta$ which is re-calculated before every step considering only active edges).
For each $\bar \alpha_k < \Delta$, we do the following:
Make a temporary copy \texttt{active$(\bar \alpha_k)$} of \texttt{active}. Let $\delta^a(i)$ be the degree of node $i$ considering only edges in \texttt{active}$(\bar \alpha_k)$. 
We sort the nodes by $\delta^a(i)$, and then for each node $i$, we sort the adjacent active edges $e=\{i,i'\} \in$ \texttt{active$(\bar \alpha_k)$} by decreasing $\delta^a(i')$. We iterate through this list and set edges to be inactive in \texttt{active$(\bar \alpha_k)$} until $\delta^a(i)\leq  \bar \alpha_k$ and then proceed to the next node. After this procedure, increasing all $\gamma_e$ in \texttt{active$(\bar \alpha_k)$} (and also increasing resp., decreasing the associated $\delta^k_e$, resp., $\alpha_k$, $k \leq \bar k$) by one gives a feasible solution with a net-change $|$\texttt{active$(\bar \alpha_k)$}$|-\bar k \cdot \bar \alpha_k$ in the objective. At each step, we choose the value $\bar \alpha_k$ giving the largest positive net-change and set the dual variables and \texttt{active} accordingly (in case of ties, we take the smallest $\bar \alpha_k$ among the values giving the largest positive net-change).
 We stop, when there is no $\bar \alpha_k$ giving positive net-change at a step $\bar k$ (observe that as in the simplified variant, the net-change gets smaller in every step).  The runtime of the extended version is $O(\Delta|V|^2|E|)$.

The following example illustrates both versions of the algorithm and also is a counter-example to Lemma 3 in \cite{fertin2015algorithmic}, which claims that a node with maximum degree gets label one in any optimal labeling.
\begin{example}
Consider the grid graph given in Figure \ref{fig:gridexample}. The graph has $|V|=9$, $|E|=12$ and $\Delta$ is four and achieved by the node ``E''. An optimal labeling $\phi^*$ is given in Figure \ref{fig:gridsolution} with $SL_{\phi^*}=30$. Note that only the nodes with label one, two, three and four contribute to the objective and node ``E'' does not have label one. Moreover, for this instance, when solving the LP-relaxation of \second\ we do not obtain $SL_{\phi^*}$ but $29.6667$, i.e., \second\ has an integrality gap for this instance.

\tikzstyle{vertex}=[circle,fill=black!15,minimum size=20pt,inner sep=0pt]
\tikzstyle{edge} = [draw,thick,-]
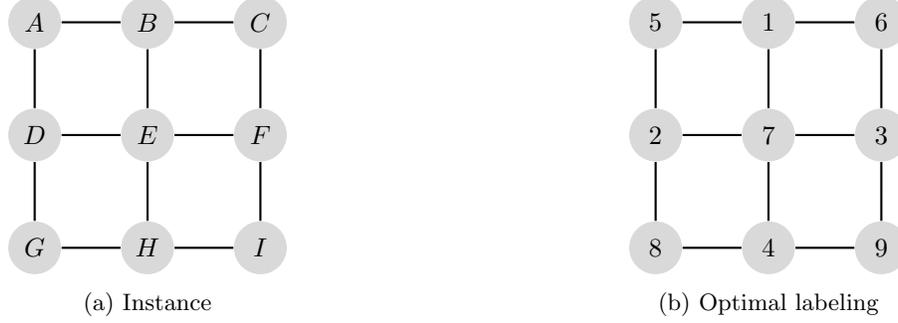
\begin{figure}[h!tb]
\begin{subfigure}[b]{.5\linewidth}
\centering
\begin{tikzpicture}[scale=1.5]
\foreach \pos/\name/\type in {{(0,2)/A/vertex}, {(1,2)/B/vertex}, {(2,2)/C/vertex},{(0,1)/D/vertex},{(1,1)/E/vertex},{(2,1)/F/vertex},{(0,0)/G/vertex},{(1,0)/H/vertex},{(2,0)/I/vertex}}
        \node[\type] (\name) at \pos {$\name$};
\foreach \source/ \dest in  
    {A/B,B/C,A/D,D/E,B/E,C/F,E/F,D/G,G/H,E/H,H/I,F/I}
    \path[edge] (\source) --  (\dest);
\end{tikzpicture}
\caption{Instance}\label{fig:gridexample}
\end{subfigure}%
\begin{subfigure}[b]{.5\linewidth}
\centering
\begin{tikzpicture}[scale=1.5]
\foreach \pos/\name/\type in {{(0,2)/5/vertex}, {(1,2)/1/vertex}, {(2,2)/6/vertex},{(0,1)/2/vertex},{(1,1)/7/vertex},{(2,1)/3/vertex},{(0,0)/8/vertex},{(1,0)/4/vertex},{(2,0)/9/vertex}}
        \node[\type] (\name) at \pos {$\name$};
\foreach \source/ \dest in  
    {A/B,B/C,A/D,D/E,B/E,C/F,E/F,D/G,G/H,E/H,H/I,F/I}
    \path[edge] (\source) --  (\dest);
\end{tikzpicture}

\caption{Optimal labeling}\label{fig:gridsolution}
\end{subfigure}
\caption{Grid-graph instance and an optimal labeling $\phi^*$ for it, $SL_{\phi^*}=3\cdot 1+3\cdot 2+3\cdot 3+3\cdot 4=30$. }\label{fig:grid}
\end{figure}

The simplified variant of our algorithm proceeds as follows:
\begin{enumerate}
\item $\alpha_1=4$, $z^D=20$ (as $|E|-\bar k \cdot \Delta=12-1 \cdot 4=8$)
\item $\alpha_2=4, \alpha_1=8$, $z^D=24$ (as $|E|-\bar k \cdot \Delta=12-2 \cdot 4=4$)
\item try $\alpha_3=4, \alpha_2=8, \alpha_1=12$, terminate as $|E|-\bar k \cdot \Delta=12-3 \cdot 4=0$
\end{enumerate}
The extended variant of our algorithm proceeds as follows:
\begin{enumerate}
\item $\alpha_1=3$, $active_{BE}=false$, $z^D=20$ (as $|active|-\bar k \cdot \delta^a=11-1 \cdot 3=8$)
\item $\alpha_2=3, \alpha_1=6$, $active_{BE}=false$, $z^D=25$ (as $|active|-\bar k \cdot \delta^a=11-2 \cdot 3=5$)
\item $\alpha_3=3, \alpha_2=6, \alpha_1=9$, $active_{BE}=false$, $z^D=27$ (as $|active|-\bar k \cdot \delta^a=11-3 \cdot 3=2$)
\end{enumerate}
Thus, the extended variant gives a better objective value of the dual solution.
\end{example}

\subsection{Paths and Cycles}

In \cite{fertin2009s,fertin2015algorithmic,fertin2017s}, the following closed formulas for $SL_{\phi^*}$ for paths $P_{n}$ and cycles $C_{n}$ with $n$ nodes are given without proof: $SL_{\phi^*}(C_n)=SL_{\phi^*}(P_{n+1})=\frac{n^2}{4}+\frac{n}{2}$ if $n$ is even and $SL_{\phi^*}(C_{n})=SL_{\phi^*}(P_{n+1})=\frac{(n+1)^2}{4}$ if $n$ is odd. We observe that paths are a special variant of caterpillar graphs, thus \cite{fertin2015algorithmic,fertin2017s} give a polynomial-time algorithm for paths. In Algorithm \ref{alg:path}, we give a linear-time algorithm for paths and in Algorithm \ref{alg:cycle} we do the same for cycles. We show that the objective value of the solution obtained by the respective algorithm (which is of course feasible for \second\, since it is a labeling) is the same value as the objective value of the solution for $D_{\second}$ produced by Algorithm \ref{alg:dual}. By strong duality this implies, that the solution is optimal and also means, that for paths and cycles, \second\ has no integrality gap.

\begin{algorithm}[h!tb]   
\DontPrintSemicolon                 
\SetKwInOut{Input}{input}\SetKwInOut{Output}{output}
\Input{instance $G=(V,E)$ of the \SLP, with $G$ being a path}
\Output{optimal labeling $\phi^*(G)$}
start from an end of the path and label each even node with the smallest unused label \;
label the odd nodes arbitrarily\;
\caption{Optimal solution algorithm for the \SLP\ when the instance is a path\label{alg:path} }
\end{algorithm}

\begin{theorem}
Suppose $G$ is a path. Let $z^P$ be the solution value obtained by Algorithm \ref{alg:path} and $z^D$ be the solution value of $(D_{\second})$ for the dual solution obtained by Algorithm \ref{alg:dual}. We have $z^P=z^D$.
\end{theorem}
\begin{proof}
We first calculate the value for $z^P$ and then the value for $z^D$, and in both cases make a case distinction whether $|V|$ is even or odd.
\begin{itemize}
\item $z^P$
\begin{itemize}
\item $|V|$ is odd\newline
By labeling every even node, we need $\frac{|V|-1}{2}$ labels to cover every edge, and each label covers two edges (the remaining labels do not contribute to the objective). Thus, $\phi=z^P$ for this labeling is 
\begin{equation*}
2\cdot\sum_{k=1}^{\frac{|V|-1}{2}}k=\frac{|V|-1}{2}\cdot\big(\frac{|V|-1}{2}+1\big)=\frac{(|V|-1)^2}{4}+\frac{|V|-1}{2}
\end{equation*}
\item $|V|$ is even\newline
By labeling every even node, we need $\frac{|V|}{2}$ labels to cover every edge, and each label except the last covers two edges, and the last one covers one edge (the remaining labels do not contribute to the objective). Thus, $\phi=z^P$ for this labeling is 
\begin{equation*}
\frac{|V|}{2}+2\cdot\sum_{k=1}^{\frac{|V|}{2}-1}k=\frac{|V|}{2}+\big(\frac{|V|}{2}-1\big)\cdot\frac{|V|}{2}=\frac{|V|^2}{4}
\end{equation*}
\end{itemize}
\item $z^D$  \newline
In paths, we have $\Delta=2$ and $|V|=|E|+1$. 
Thus, Algorithm \ref{alg:dual} stops at step $\bar k$, when $(|V|-1)-2 \cdot\bar k \leq 0 $
\begin{itemize}
\item $|V|$ is odd\newline
 $(|V|-1)-2 \cdot\bar k \leq 0 $ when $\bar k=\frac{|V|-1}{2}$. As at each step $k$ before termination, the net-change in $z^D$ was $(|V|-1)-2 \cdot k$ and the initial dual solution has value $|E|$, we obtain 
 \begin{align*}
 & z^D=\sum_{k=0}^{\frac{|V|-1}{2}-1}\Big((|V|-1)-2 \cdot k\Big) \\
& =(|V|-1)\cdot\frac{|V|-1}{2}-2\cdot\sum_{k=0}^{\frac{|V|-1}{2}-1}k\\
&=\frac{(|V|-1)^2}{2}-\frac{|V|-1}{2}\cdot(\frac{|V|-1}{2}-1)=\frac{(|V|-1)^2}{4}+\frac{|V|-1}{2}.
  \end{align*}
 
 \item $|V|$ is even\newline
  $(|V|-1)-2 \cdot\bar k \leq 0 $ when $\bar k=\frac{|V|}{2}$. As at each step $k$ before termination, the net-change in $z^D$ was $(|V|-1)-2 \cdot\bar k$ and the initial dual solution has value $|E|$, we obtain
  \begin{align*}
  & z^D=\sum_{k=0}^{\frac{|V|}{2}-1}\Big((|V|-1)-2 \cdot k\Big)\\
  & =(|V|-1)\cdot\frac{|V|}{2}-2\cdot\sum_{k=0}^{\frac{|V|}{2}-1}k\\   
  & =\frac{|V|^2-|V|}{2}-\frac{|V|}{2}\cdot(\frac{|V|}{2}-1)=\frac{|V|^2}{4}.  
  \end{align*}
\end{itemize}
\end{itemize}
\end{proof}

\begin{algorithm}[h!tb]   
\DontPrintSemicolon                 
\SetKwInOut{Input}{input}\SetKwInOut{Output}{output}
\Input{instance $G=(V,E)$ of the \SLP, with $G$ being a cycle}
\Output{optimal labeling $\phi^*(G)$}
start with an arbitrary node and label each even node with the smallest unused label \;
label the odd nodes arbitrarily\;
\caption{Optimal solution algorithm for the \SLP\ when the instance is a cycle\label{alg:cycle} }
\end{algorithm}

\begin{theorem}
Suppose $G$ is a cycle. Let $z^P$ be the solution value obtained by Algorithm \ref{alg:cycle} and $z^D$ be the solution value of $(D_{\second})$ for the dual solution obtained by Algorithm \ref{alg:dual}. We have $z^P=z^D$.
\end{theorem}
\begin{proof}
Similar to the proof for paths.
\end{proof}

\subsection{Perfect $n$-ary trees}

A perfect $n$-ary tree is a rooted tree, where all internal (i.e., non-leaf) nodes have $k$ children and all leaf nodes are at the same depth $d$ \cite{perfect}. Algorithm \ref{alg:ktree} gives a linear-time algorithm to solve the \SLP\ to optimality, and Theorem \ref{thm:nary} shows that \second\ has no integrality gap for such trees and that $SL_{\phi^*}=\frac{(|V|-1)^2}{2\cdot(n+1)}+\frac{|V|-1}{2}$ in case $d$ is odd, and $SL_{\phi^*}=\frac{(|V|-1-n)^2}{2\cdot(n+1)}+\frac{n\cdot(|V|-1-n)}{n+1}+\frac{|V|-1+n}{2} $ in case $d$ is even.

\begin{algorithm}[h!tb]   
\DontPrintSemicolon                 
\SetKwInOut{Input}{input}\SetKwInOut{Output}{output}
\Input{instance $G=(V,E)$ of the \SLP, with $G$ being a perfect $n$-ary tree with depth $d$}
\Output{optimal solution $\phi^*(G)$ of the \SLP}
\If{$d$ is odd}
{
take the nodes at even depths and label them with the smallest unused labels \;
arbitrarily label the remaining nodes\;
}
\Else
{
take the nodes at odd depths except the root node and label them with the smallest unused labels \;
label the root node with the smallest unused label \;
arbitrarily label the remaining nodes\;
}

\caption{Optimal solution algorithm for the \SLP\ when the instance is a perfect $n$-ary tree\label{alg:ktree} }
\end{algorithm}

\begin{theorem}\label{thm:nary}
Suppose $G$ is a perfect $n$-ary tree with depth $d$. Let $z^P$ be the solution value obtained by Algorithm \ref{alg:ktree} and $z^D$ be the solution value of $(D_{\second})$ for the dual solution obtained by Algorithm \ref{alg:dual}. We have $z^P=z^D$.
\end{theorem}
\begin{proof}
We first calculate the value for $z^P$ and then the value for $z^D$, and in both cases make a case distinction whether $d$ is even or odd.
\begin{itemize}
\item $z^P$
\begin{itemize}
\item $d$ is odd \newline
The nodes on the even depths cover every edge in the graph exactly once and each node covers $n+1$ edges (the edges going to its $n$ children nodes, and the edge going to its parent node). Thus only these nodes contribute to the objective function and $\phi=z^P$ for this labeling is
  \begin{align*}
 (n+1) \cdot \sum_{k=1}^{\frac{|V|-1}{n+1}} k
 =\frac{(n+1)}{2}\frac{|V|-1}{n+1} \cdot \big(\frac{|V|-1}{n+1}+1\big)
 =\frac{(|V|-1)^2}{2\cdot(n+1)}+\frac{|V|-1}{2}
   \end{align*}
\item $d$ is even \newline
The nodes on the odd depths cover every edge in the graph exactly once and each node except the root node covers $n+1$ edges (the edges going to its $n$ children nodes, and the edge going to its parent node), while the root node covers $n$ edges (the edges going to its $n$ children nodes). Thus only these nodes contribute to the objective function and $\phi=z^P$ for this labeling is
\begin{align*}
& (n+1) \cdot \sum_{k=1}^{\frac{|V|-1-n}{n+1}}k+n \cdot \big(\frac{|V|-1-n}{n+1}+1 \big) =\frac{n+1}{2} \cdot \frac{|V|-1-n}{n+1}\cdot \big( \frac{|V|-1-n}{n+1}+1\big) +n \cdot \big(\frac{|V|-1-n}{n+1}+1 \big) \\
&=\frac{(|V|-1-n)^2}{2\cdot(n+1)}+\frac{|V|-1-n}{2}+\frac{n\cdot(|V|-1-n)}{n+1}+n=\frac{(|V|-1-n)^2}{2\cdot(n+1)}+\frac{n\cdot(|V|-1-n)}{n+1}+\frac{|V|-1+n}{2}
\end{align*}

\end{itemize}
\item $z^D$ \newline
We have $\Delta=n+1$ and $|V|=|E|+1$. Thus, Algorithm \ref{alg:dual} stops at step $\bar k$, when $(|V|-1)-(n+1)\cdot \bar k \leq 0$
\begin{itemize}
\item $d$ is odd \newline
 $(|V|-1)-(n+1)\cdot \bar k \leq 0$ when $\bar k=\frac{|V|-1}{n+1}$.  As at each step $k$ before termination, the net-change in $z^D$ was $(|V|-1)-(n+1) \cdot k$ and the initial dual solution has value $|E|$, we obtain 
 \begin{align*}
 &z^D=\sum_{k=0}^{\frac{|V|-1}{n+1}-1} \Big((|V|-1)-(n+1) \cdot k\Big)\\
 &=(|V|-1) \cdot \frac{|V|-1}{n+1} - (n+1) \sum_{k=0}^{\frac{|V|-1}{n+1}-1} k \\
 &=\frac{(|V|-1)^2}{n+1}- \frac{n+1}{2}\big(\frac{|V|-1}{n+1}-1\big)\cdot \frac{|V|-1}{n+1}\\
 &=\frac{(|V|-1)^2}{2\cdot(n+1)}+\frac{|V|-1}{2}
  \end{align*}
\item $d$ is even \newline
$(|V|-1)-(n+1)\cdot \bar k \leq 0$ when $\bar k=\frac{|V|-1-n}{n+1}+1$. As at each step $k$ before termination, the net-change in $z^D$ was $(|V|-1)-(n+1) \cdot k$ and the initial dual solution has value $|E|$, we obtain 
 \begin{align*}
 &z^{D}=\sum_{k=0}^{\frac{|V|-1-n}{n+1}}\Big((|V|-1)-(n+1) \cdot k\Big)\\
 &=(|V|-n+n-1) \cdot \Big(\frac{|V|-1-n}{n+1}+1\Big)- (n+1) \sum_{k=0}^{\frac{|V|-1-n}{n+1}} k \\
  &=\frac{(|V|-1-n)^2}{n+1}+\frac{n\cdot(|V|-1-n)}{n+1}+ |V|-1 - \frac{n+1}{2}\cdot\Big(\frac{|V|-1-n}{n+1}+1\Big)\cdot \frac{|V|-1-n}{n+1} \\
 &=\frac{(|V|-1-n)^2}{2\cdot(n+1)}
 +\frac{n\cdot(|V|-1-n)}{n+1}+ |V|-1 -\frac{|V|-1-n}{2}\\
 &=\frac{(|V|-1-n)^2}{2\cdot(n+1)}
  +\frac{n\cdot(|V|-1-n)}{n+1}+\frac{|V|-1+n}{2}
   \end{align*}
  
\end{itemize}
\end{itemize}

\end{proof}

\section{Additional solution approaches \label{sec:alternative}}

In this section, we describe a Lagrangian heuristic and a constraint programming formulation
for the \SLP.

\subsection{Lagrangian heuristic \label{sec:lagheur}}

Let

\begin{align}
\quad z^* = \min\left\{\mathbf{c}^T\mathbf{x}\mid A\mathbf{x}\leq \mathbf{b},\; H\mathbf{x} \leq \mathbf{h}\;\text{and}\;\mathbf{x}\in \mathbb Z^n \right\},\tag{COP}\label{eq:co}
\end{align}

be combinatorial optimization problem (COP), which can be formulated as MIP 
with a set of \emph{easy} constraints $A\mathbf{x}\leq \mathbf{b}$ and \emph{complicating} constraints $H\mathbf{x} \leq \mathbf{h}$. \emph{Easy} and \emph{complicating} constraints means, that the problem \eqref{eq:co} without constraints $H\mathbf{x} \leq \mathbf{h}$ is easy to solve, e.g., using a combinatorial algorithm.
Lagrangian relaxation (see,~e.g.,~\citep{fisher1981lagrangian,wolsey1998integer}) is an attractive way to solve solve such problems.

Let $\boldsymbol \lambda \geq 0$ be a vector of \emph{dual multipliers} for $H\mathbf{x} \leq \mathbf{h}$. The Lagrangian relaxation of \eqref{eq:co} for a given $\boldsymbol \lambda$ is defined as 
\begin{align}
{z_R}(\boldsymbol{\lambda}) = \min\left\{\left(\mathbf{c}^T + {\boldsymbol{\lambda}}^T {H}\right)\mathbf{x} - {\boldsymbol{\lambda}}^T {\mathbf{h}} \mid A\mathbf{x} \leq \mathbf{b}\;\text{and}\;\mathbf{x}\in\mathbb Z^n \right\}.\label{eq:relaxco}\tag{LR}
\end{align}
The value ${z_R}(\boldsymbol{\lambda})$ gives a lower bound for the objective of \eqref{eq:co}, i.e., $z^*\geq {z_R}(\boldsymbol{\lambda})$, and to find the best lower bound, a maximization problem in $\boldsymbol{\lambda}$, called the \emph{Lagrangian dual problem}, 
\begin{align}
\max_{\boldsymbol{\lambda}\geq 0} {z_R}(\boldsymbol{\lambda})\label{eq:lagdual} \tag{LD}
\end{align}
is solved. 
We use a subgradient method to solve \eqref{eq:lagdual}. Within the subgradient method, the Lagrangian relaxation \eqref{eq:relaxco} gets iteratively solved for different multipliers $\boldsymbol{\lambda}$ and the best value of ${z_R}(\boldsymbol{\lambda})$ is taken as lower bound $z_{LB}$. The value of $\boldsymbol{\lambda}^{t+1}$ at iteration $t+1$ of the subgradient method,
is obtained from the current solution $\mathbf x^t$ with the help of a subgradient $\mathbf g^t$, which is calculated as $\mathbf g^t=\mathbf h-H\mathbf x^t$. We use a standard variant for updating the multipliers, which we describe in the following (see,~e.g.,~\citep{conforti2014integer,wolsey1998integer} for more details): The value of $\boldsymbol{\lambda}^{t+1}$ is calculated as $\boldsymbol{\lambda}^{t+1}=\max\{0,\boldsymbol{\lambda}^{t}-\mu^t \mathbf g^t\}$, with the step-size $\mu^t=\beta \frac{z^I-z_R\left(\boldsymbol{\lambda}^{t} \right) }{\left\| \mathbf{g}^t\right\|}$, where $z^I$ is the value of the best feasible solution found so far, and $\beta$ a given parameter in $(0,2]$. We initialize $\beta$ with two, and if there are $\tau=7$ iterations without an improvement of ${z_R}(\boldsymbol{\lambda}^t)$, we set $\beta=\beta/2$. We have an iteration limit of 500 iterations as stopping criterion, moreover, we stop when either $z_I-z_{LB}<1$, $\mu^t<10^{-5}$ or $||g^t||<10^{-6}$.


To apply Lagrangian relaxation to the \SLP, we take formulation \second, and relax constraints \eqref{eq:link}, let $\delta^k_e\geq 0$ be the associated dual multipliers. We obtain the following relaxed problem $(LR_{\second})$.

\begin{align} 
(LR_{\second})  \quad\min\quad \sum_{e \in E} \sum_{1 \leq k \leq |V|} (k+\delta^k_e) & d^k_e-\sum_{i \in V}& \sum_{1 \leq k \leq |V|} \Big(\sum_{e \in E:i \in e}    \delta^k_e \Big) x^k_i  \label{eq:obj3} \tag{$(LR_{\second}).1$} \\
&\eqref{eq:a-sum},& \eqref{eq:b-sum} \notag \\
\sum_{1 \leq k \leq |V|} d^k_e&=1 &\quad \forall e \in E \label{eq:e-sum3} \tag{$(LR_{\second}).1$} \\
 x^k_{i} & \in\{0,1\}&\;\quad 1 \leq i,k \leq |V|
 \notag \\
  d^k_{e} & \in\{0,1\}&\;\quad \forall e \in E, 1 \leq k \leq |V|
   \notag \\
   \delta^k_e & \geq 0&\;\quad \forall e \in E, 1 \leq k \leq |V|  \notag
\end{align}

For a fixed set of multipliers $\bar \delta^k_e$, it is easy to see, that the problem decomposes into a maximum assignment problem in the $x$-variables, i.e.,
\begin{align} 
\max\quad &\sum_{i \in V} \sum_{1 \leq k \leq |V|} \Big(\sum_{e \in E:i \in e} \bar \delta^k_e \Big) x^k_i  \notag \\
&\eqref{eq:a-sum}, \eqref{eq:b-sum} & \notag \\
 x^k_{i} & \in\{0,1\}&\;\quad 1 \leq i,k \leq |V|
 \notag 
\end{align}
and for each edge $e \in E$ into a simple problem, where the index $k$ with minimum objective coefficient needs to be selected, i.e., 
\begin{align} 
\min\quad \sum_{1 \leq k \leq |V|}&(k+ \bar \delta^k_e) d^k_e \notag \\
\sum_{1 \leq k \leq |V|} d^k_e&=1 \notag \\
  d^k_{e} & \in\{0,1\}\;\quad 1 \leq k \leq |V|  \notag 
\end{align}

Every solution of the assignment problem during the course of the subgradient algorithm gives a feasible labeling $\phi_H$ and we use the local-search phase of Algorithm \ref{alg:primal} and try to improve the obtained $\phi_H$. 
%
We initialize the multipliers $\delta^k_e$ with the values obtained by the extended version of the dual heuristic described in Algorithm \ref{alg:dual}. We also use inequalities \eqref{eq:valid}, by adding and relaxing all of them for the increasing subsets of the labels up to $\{1,2,\ldots, |V|-1\}$.
To generate an initial starting solution, we use Algorithm \ref{alg:primal}.

\subsection{A Constraint Programming formulation\label{sec:cp}}

Let $y_i$ be an integer variable denoting the label of node $i\in V$. Using the $min$-constraint and the $alldifferent$-constraint, \SLP\ can be formulated as constraint programming problem as follows

\begin{align}
\min \sum_{e=\{i,i'\}\in E} \min\{y_i,y_{i'}\}& \tag{C.1} \\
alldifferent(y) &\tag{C.2} \\
y_i \in \{1,2,\ldots, |V|\},& \quad \forall i \in V \tag{C.3}
\end{align}

%
%
%


\section{Computational results \label{sec:res}}

All approaches were implemented in C++, the branch-and-\blue{cut} frameworks were implemented using CPLEX 12.7, and the constraint programming solver was also implemented using the same CPLEX version. To solve the assignment problems arising as subproblems in the Lagrangian relaxation, and also in the primal heuristic of the branch-and-\blue{cut} framework, we used the algorithm available at \url{http://dlib.net/}, which implements the Hungarian method \cite{kuhn1955hungarian}. The runs were made \blue{over} an Intel Xeon E5 v4 CPU with 2.2 GHz and 3GB memory and using a single thread. 
The timelimit for a run was set to 600 seconds and all CPLEX-settings were left at default, except when solving \first, details of the changed settings in this case are given in Section \ref{sec:resgen}.

\subsection{Computational tests for special graph classes}

First, we are interested, if there may be additional graph classes, for which the LP-relaxation of our MIPs exhibits no integrality gap.
Thus we generated the following sets of graphs using the graph generators of the \texttt{NetworkX}-package \cite{hagberg-2008-exploring}. \blue{The instances are available at \url{https://msinnl.github.io/pages/instancescodes.html}}.

\begin{itemize}
\item \texttt{grid}: We generated $|V|\times|V|$ grid graphs for $|V| \in \{3,4,\ldots,12\}$ using the \texttt{grid\_graph}-function.
\item \texttt{bipartite graphs}: We generated $(n,m)$ bipartite graphs, with edge probability $p$, for $(n,m) \in \{(5,5),(10,10),(15,15),(20,20),(25,25)\}$ and $p \in \{0.25,0.5\}$ using the \texttt{random\_graph}-function of the \texttt{bipartite}-module.
\item \texttt{caterpillar}: A caterpillar is a tree, which reduces to a path, when all leaf nodes are deleted. In \cite{fertin2015algorithmic,fertin2017s} a polynomial-time algorithm for the \SLP\ on caterpillar graphs is given.
We generated caterpillars with the expected number of nodes in the path in $\{10,20,30,40,50\}$ and the probability of adding edges to the underlying path in $\{0.25,0.5\}$ using the \texttt{random\_lobster}-function (and setting the probability of adding a second level edges to zero).
\item \texttt{lobster}: A lobster graph is a tree, which reduces to a caterpillar, when all leaf nodes are deleted. We generated lobsters with the expected number of nodes in the path in $\{10,20,30,40,50\}$ and the probability of adding edges to the underlying path in $\{0.25,0.5\}$ and also the same probability of adding a second level of edges using the \texttt{random\_lobster}-function
\item \texttt{tree}: We generated trees with $|V|=\{10,15,20,25,30,35,40,50,75,100\}$ using the \texttt{random\_tree}-function
\end{itemize}


Tables \ref{ta:grid} to \ref{ta:tree} show the value of the LP-relaxation (columns $z_{LP}$) of $\second$ (recall that both formulations have the same strength, as shown in Section \ref{sec:strength}) and the value of the optimal solution (columns $z^*$), obtained by solving the MIP, a bold entry in $z_{LP}$ means there is no gap. Only for three of the tested graphs, there is a gap. Two of these three are grid graphs, and one is a lobster graph. Hence, for grid graphs and trees (as a lobster is a special case of a tree), $\second$ has an integrality gap. Interestingly, while for a lobster there is a gap, for the instances in \texttt{tree}, there is no gap. Thus, clearly caution is advised when drawing conclusions from these results regarding the integrality gap of $\second$. \blue{Regarding the instances without a gap, we observe that for caterpillars the problem is known to be polynomial-time solvable.} For bipartite graphs, it is also interesting to see that for the tested instances, there is no gap, as the complexity of the \SLP\ is still unknown for these graphs. Unfortunately, we were unable to prove results regarding the integrality gap for both graph types \blue{when using a similar approach as we did in Section \ref{sec:analyze} for paths, cycles and $n$-ary trees}. \blue{For caterpillars, this was due to the dual side, i.e., we found no rule for updating the dual multipliers within the dual ascent which gave a dual optimal solution. For bipartite graphs, this was due to both dual and also primal reasons, i.e., there is no (polynomial time) algorithm known to solve the problem (and we were not able to design one).}

\begin{table}[ht]
\caption{Values of LP-relaxation and optimal solution for special graphs \label{ta:special}}
\begin{subtable}[b]{.3\linewidth}
\centering
\begingroup\scriptsize
\begin{tabular}{l|rr}
  \toprule
  name & $z^*$ & $z_{LP}$\\ \midrule
  gridgraph3 & 30 & 29.67 \\ 
  gridgraph4 & 96 & \textbf{96} \\ 
  gridgraph5 & 242 & 241.67 \\ 
  gridgraph6 & 514 & \textbf{514} \\ 
  gridgraph7 & 972 & \textbf{972} \\ 
  gridgraph8 & 1692 & \textbf{1692} \\ 
  gridgraph9 & 2750 & \textbf{2750} \\ 
  gridgraph10 & 4254 & \textbf{4254} \\ 
    gridgraph11 & 6296 & \textbf{6296} \\ 
    gridgraph12 & 9016 & \textbf{9016} \\ 
   \bottomrule
\end{tabular}
\caption{Grid graphs \label{ta:grid}} 
\endgroup
\end{subtable}
\begin{subtable}[b]{.3\linewidth}
\centering
\begingroup\scriptsize
\begin{tabular}{l|rr}
  \toprule
  name & $z^*$ & $z_{LP}$\\ \midrule
    bipartite5-5-0.25 & 6 & \textbf{6} \\ 
    bipartite5-5-0.5 & 19 & \textbf{19} \\ 
bipartite10-10-0.25 & 85 & \textbf{85} \\ 
  bipartite10-10-0.5 & 211 & \textbf{211} \\ 
  bipartite15-15-0.25 & 254 & \textbf{254} \\ 
  bipartite15-15-0.5 & 707 & \textbf{707} \\ 
  bipartite20-20-0.25 & 880 & \textbf{880} \\ 
  bipartite20-20-0.5 & 1893 & \textbf{1893} \\ 
  bipartite25-25-0.25 & 1837 & \textbf{1837} \\ 
  bipartite25-25-0.5 & 3641 & \textbf{3641} \\ 
   \bottomrule
\end{tabular}
\caption{Bipartite graphs \label{ta:bipartite}} 
\endgroup
\end{subtable}
\begin{subtable}[b]{.3\linewidth}
\centering
\begingroup\scriptsize
\begin{tabular}{l|rr}
  \toprule
  name & $z^*$ & $z_{LP}$\\ \midrule
caterpillar10-0.25 & 41 & \textbf{41} \\ 
  caterpillar10-0.5 & 57 & \textbf{57} \\ 
  caterpillar20-0.25 & 434 & \textbf{434} \\ 
  caterpillar20-0.5 & 503 & \textbf{503} \\ 
  caterpillar30-0.25 & 306 & \textbf{306} \\ 
  caterpillar30-0.5 & 397 & \textbf{397} \\ 
  caterpillar40-0.25 & 508 & \textbf{508} \\ 
  caterpillar40-0.5 & 583 & \textbf{583} \\ 
  caterpillar50-0.25 & 682 & \textbf{682} \\ 
  caterpillar50-0.5 & 932 & \textbf{932} \\ 
   \bottomrule
\end{tabular}
\endgroup
\caption{Caterpillar graphs \label{ta:caterpillar}} 
\end{subtable}

\centering
\begin{subtable}[b]{.3\linewidth}
\begingroup\scriptsize
\begin{tabular}{l|rr}
  \toprule
  name & $z^*$ & $z_{LP}$\\ \midrule
lobster10-0.25-0.25 & 41 & \textbf{41} \\ 
  lobster10-0.5-0.5 & 72 & \textbf{72} \\ 
  lobster20-0.25-0.25 & 492 & \textbf{492} \\ 
  lobster20-0.5-0.5 & 793 & \textbf{793} \\ 
  lobster30-0.25-0.25 & 361 & \textbf{361} \\ 
  lobster30-0.5-0.5 & 551 & \textbf{551} \\ 
  lobster40-0.25-0.25 & 533 & \textbf{533} \\ 
  lobster40-0.5-0.5 & 766 & \textbf{766} \\ 
  lobster50-0.25-0.25 & 737 & \textbf{737} \\ 
  lobster50-0.5-0.5 & 1267 & 1266.5 \\ 
   \bottomrule
\end{tabular}
\endgroup
\caption{Lobster graphs \label{ta:lobster}} 
\end{subtable}
\begin{subtable}[b]{.3\linewidth}
\begingroup\scriptsize
\begin{tabular}{l|rr}
  \toprule
  name & $z^*$ & $z_{LP}$\\ \midrule
tree10 & 19 & \textbf{19} \\ 
  tree15 & 40 & \textbf{40} \\ 
  tree20 & 77 & \textbf{77} \\ 
  tree25 & 129 & \textbf{129} \\ 
  tree30 & 152 & \textbf{152} \\ 
  tree35 & 186 & \textbf{186} \\ 
  tree40 & 235 & \textbf{235} \\ 
  tree50 & 386 & \textbf{386} \\ 
  tree75 & 943 & \textbf{943} \\ 
    tree100 & 1692 & \textbf{1692} \\ 
   \bottomrule
\end{tabular}
\endgroup
\caption{Tree graphs \label{ta:tree}} 
\end{subtable}
\end{table}

\newcommand{\HB}{\texttt{HB}}
\newcommand{\RND}{\texttt{RND}}
\newcommand{\COL}{\texttt{COL}}
\newcommand{\sm}{\texttt{sm}}
\newcommand{\lag}{\texttt{lg}}

\subsection{Computational tests on general graphs \label{sec:resgen}}

In this section, we are interested in investigating the computational effectiveness of the proposed solution approaches for the \SLP\ when dealing with general graphs. The considered graphs are as follows:

\begin{itemize}
\item \HB: Graphs from the Harwell-Boeing Sparse Matrix Collection, which ``is a set of standard test matrices arising from problems in linear systems, least squares, and eigenvalue calculations from a wide variety of scientific and engineering disciplines'', see \url{https://math.nist.gov/MatrixMarket/collections/hb.html}. These graphs have for example be used when testing approaches for graph bandwidth problems, which are another type of labeling problem \cite{duarte2011grasp,rodriguez2015tabu}. Details of the number of nodes and edges of the individual graphs are given in the results-tables in columns $|V|$ and $|E|$. For testing, we partitioned the set into small and large graphs, denoted by the suffixes \sm\ and \lag\ in the names of the set. Graphs with $|V|\leq 118$ are considered as small, and all others are considered as large. There are 28 graphs \blue{in total}, and 12 are large. \blue{The instances are available at \url{https://www.researchgate.net/publication/272022702_Harwell-Boeing_graphs_for_the_CB_problem}}.
\item \RND: Random graphs with $|V|$ nodes and $|E|$ edges, generated using the \texttt{gnm\_random\_graph}-function of the \texttt{NetworkX}-package. We generated five graphs each for the following $(|V|,|E|)$-pairs: $(|V|,|E|) \in \{(50,100),(50,150),(50,200),(500,1000),(500,1500),(500,2000)\}$, the graphs with 50 nodes are considered as small (\sm) and the ones with 500 nodes as large (\lag). \blue{The instances are available at \url{https://msinnl.github.io/pages/instancescodes.html}}.
\end{itemize}

\newcommand{\FO}{\texttt{F1}}
\newcommand{\FT}{\texttt{F2}}
\newcommand{\LAG}{\texttt{LAG}}
\newcommand{\CP}{\texttt{CP}}
\newcommand{\HEUR}{\texttt{HEUR}}
\newcommand{\DUAL}{\texttt{DUAL}}
\newcommand{\FOI}{\texttt{F1I}}
\newcommand{\FOIB}{\texttt{F1IB}}
\newcommand{\FTV}{\texttt{F2V}}
\newcommand{\FTVB}{\texttt{F2VB}}

We first tested different configurations of our MIP approaches:
\begin{itemize}
\item \FO: The branch-and-\blue{cut} based on \first\ with the starting and primal heuristic described in Section \ref{sec:heur}; without the branching-scheme and initialization-scheme described in Section \ref{sec:further}
\item \FOI: Setting \FO\ and the initialization-scheme as described in Section \ref{sec:further}
\item \FOIB: Setting \FOI\ and the branching-scheme as described in Section \ref{sec:further}
\item \FT: The branch-and-\blue{cut} based on \second\ with the starting and primal heuristic described in Section \ref{sec:heur}; without the branching-scheme described in Section \ref{sec:further} and the valid inequalities \eqref{eq:valid}
\item \FTV: Setting \FT\ and the valid inequalities \eqref{eq:valid}
\item \FTVB: Setting \FTV\ and the branching-scheme as described in Section \ref{sec:further}
\end{itemize}

Figures \ref{fig:runtime} and \ref{fig:opt} show the runtime and optimality gap for these settings and the small instances (i.e., \HB-\sm\ and \RND-\sm). The optimality gap $g[\%]$ is calculated as $100 \cdot (z^D-z^*)/z^*$, where $z^D$ \blue{is} the value of the dual bound and $z^*$ is the value of the best solution found by the setting.

\begin{figure}[h!tb]
	
\centering
\begin{subfigure}[b]{.8\linewidth}
\centering \includegraphics[width=.99\linewidth]{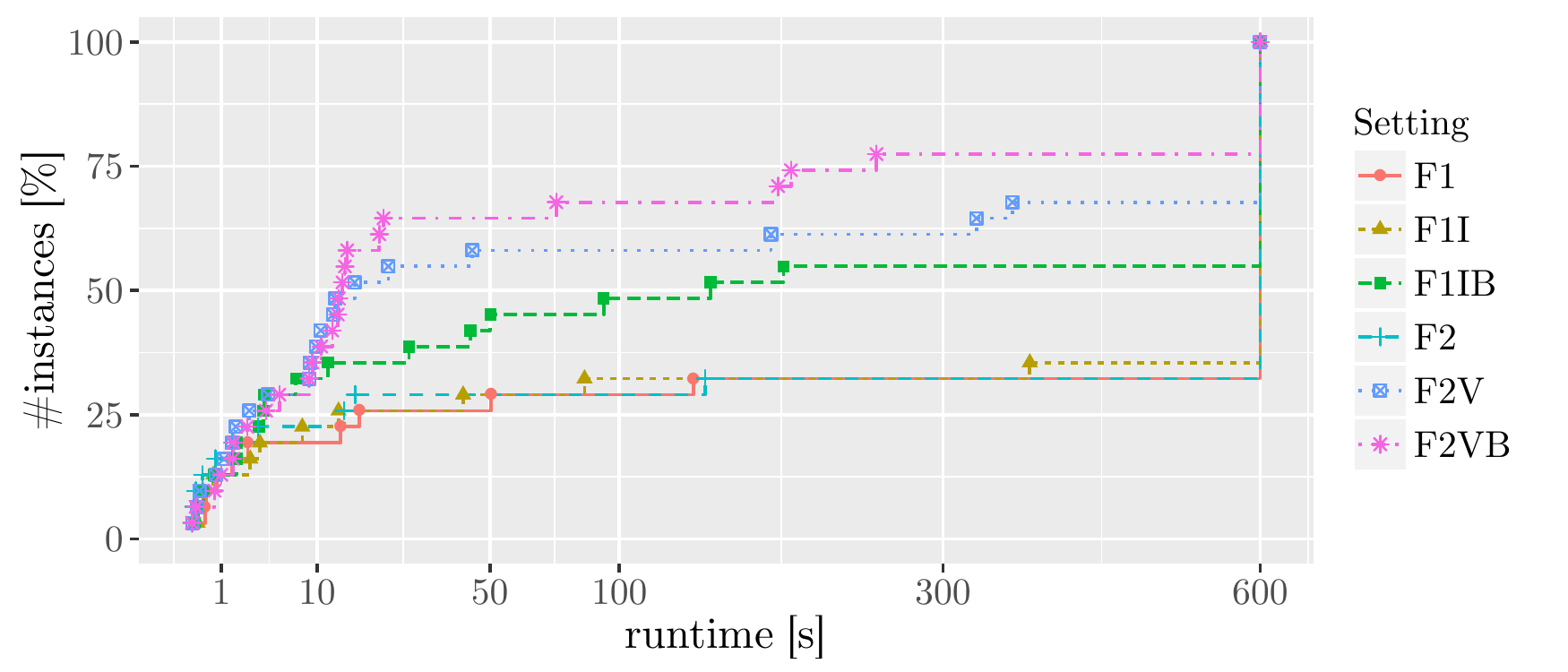}
\caption{Runtime}\label{fig:runtime}
\end{subfigure}

\centering
\begin{subfigure}[b]{.8\linewidth}
\centering \includegraphics[width=.99\linewidth]{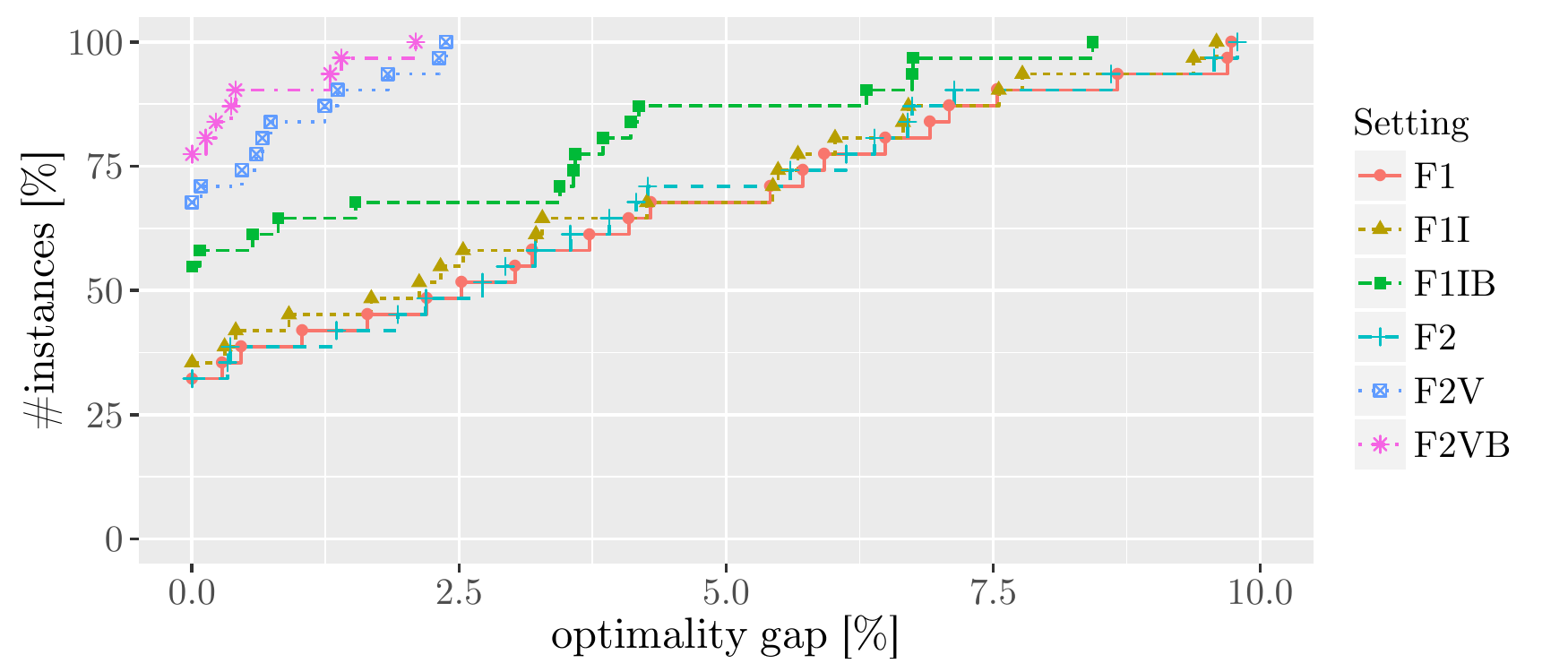}
\caption{Optimality gap}\label{fig:opt}
\end{subfigure}

\caption{Runtime and optimality gap for the small instances and different settings. \blue{Note that Figure \ref{fig:runtime} has a logarithmic scale for the runtime.}}\label{fig:small}
\end{figure}

In the figures, we see that \FTVB\ is the most effective setting, managing to solve about 75\% of the instances within the given timelimit, and the maximum gap is under 2.5\%. Settings \FTV\ and \FOIB\ also manage to solve over 50\% of the instances. In terms of optimality gap, \FTVB\ and \FTV\ are quite similar, while \FOIB\ has gaps up to 7.5\%. The remaining settings are also quite similar, with about 37.5\% of solved instances and gap of up to 10\%. Thus, the valid inequalities~\eqref{eq:valid} (used in \FTVB\ and \FTV) and the branching-scheme (used in \FTVB\ and \FOIB) seem to be quite helpful for solving the problem.  

\blue{In the remainder of this section, aside from \FOIB\ and \FTVB, we also report results for the following approaches:
\begin{itemize}
\item \DUAL: The extended version of Algorithm \ref{alg:dual}
\item \LAG: The Lagrangian heuristic described in Section \ref{sec:lagheur}
\item \CP: The Constraint Programming approach described in Section \ref{sec:cp}
\item \HEUR: The starting heuristic described in Section \ref{sec:heur}
\end{itemize}
 }

Tables \ref{ta:hblsm} and \ref{ta:rndsm} give detailed results for these approaches. We note that the values $z_D$ of \DUAL\ and $z^*$ of \HEUR\ are always worse than the corresponding values obtained by \FOIB, \FTVB, \LAG, as the latter are initialized based on \DUAL\ and \HEUR\ (see Sections \ref{sec:heur}, \ref{sec:further}, \ref{sec:lagheur}). For \FOIB\ and \FTVB\ we also report the LP-bound at the root node in column $z_R$. Bold entries in the columns mean that the method found the best values of $z_D$ and $z^*$ amongst all considered methods. We see that on the primal side (i.e., $z^*$), \FOIB, \FTVB, \LAG\ all find the best solution value for all instances, while \CP\ and the stand-alone heuristic \HEUR\ only find the best solution value for some instances. Looking at the runtime, \LAG\ takes at most five seconds, compared with the MIP-based methods \FOIB\ and \FTVB, which sometimes run until the timelimit, thus \LAG\ can be a good option, if one wants to find a good solution value quickly. Looking at the dual side, we see that \FTVB\ provides the best lower bound for all instances and the root-bound $z_R$ is better than the root-bound of \FOIB\ for nearly all instances, so the valid inequalities \eqref{eq:valid} seem quite helpful. \LAG, which also uses \eqref{eq:valid} sometimes has better bounds than $z_R$ of \FOIB\ (the MIP-approaches also benefit from the general-purpose cuts of CPLEX). The bounds provided by \DUAL\ are about 10\% worse than $z_R$ of \FOIB\ for most of the instances, however, for some they are also considerably smaller (e.g., instance \texttt{D-bcsstk01}), and for instance \texttt{rgg010} \DUAL\ even manages to find the LP-bound. Regarding the runtime for the MIP-approaches, the size of the instance seems to be important, as for all except one of the largest instances (i.e., the ones with $|E|=200$) of \RND-\sm, both \FOIB\ and \FTVB\ terminate due to the timelimit, and also for \HB-\sm, there is a connection between termination due to timelimit and number of edge\blue{s} of an instance.

\setlength{\tabcolsep}{4pt}

\begin{landscape}
\begin{table}[h!tb]
\centering
\caption{Results for \blue{instances of the class} \HB-\sm \label{ta:hblsm}} 
\begingroup\scriptsize
\begin{tabular}{lll|rrrrr|rrrrr|rrrr|rr|rr|rr}
  \toprule
   & & & \multicolumn{5}{c}{\FOIB} & \multicolumn{5}{c}{\FTVB}  & \multicolumn{4}{c}{\LAG} & \multicolumn{2}{c}{\CP} & \multicolumn{2}{c}{\HEUR} & \multicolumn{2}{c}{\DUAL}\\ name & $|V|$ & $|E|$ & $z_D$ & $z^*$ & $g [\%]$ & $t [s]$ & $z_R$ & $z_D$ & $z^*$ & $g [\%]$ & $t [s]$  & $z_R$ & $z_D$ & $z^*$ & $g [\%]$ & $t [s]$ & $z^*$ & $t [s]$ & $z^*$ & $t [s]$ & $z_D$ & $t [s]$ \\ \midrule
A-pores-1 & 30 & 103 & 811.4 & \textbf{818} & 0.81 & \textit{TL} & 744.00 & \textbf{818} & \textbf{818} & 0.00 & 14 & 807.77 & 756.27 & \textbf{818} & 7.55 & 0 & \textbf{818} & \textit{TL} & 832 & 0 & 723 & 0 \\ 
  B-ibm32 & 32 & 90 & \textbf{651} & \textbf{651} & 0.00 & 2 & 626.50 & \textbf{651} & \textbf{651} & 0.00 & 2 & 651.00 & 613.16 & \textbf{651} & 5.81 & 0 & \textbf{651} & \textit{TL} & \textbf{651} & 0 & 584 & 0 \\ 
  can-24 & 24 & 68 & 418.5 & \textbf{425} & 1.53 & \textit{TL} & 374.50 & \textbf{425} & \textbf{425} & 0.00 & 12 & 417.48 & 390.88 & \textbf{425} & 8.03 & 0 & \textbf{425} & \textit{TL} & \textbf{425} & 0 & 359 & 0 \\ 
  C-bcspwr01 & 39 & 46 & \textbf{332} & \textbf{332} & 0.00 & 0 & 332.00 & \textbf{332} & \textbf{332} & 0.00 & 0 & 332.00 & 328.06 & \textbf{332} & 1.19 & 0 & \textbf{332} & \textit{TL} & \textbf{332} & 0 & 328 & 0 \\ 
  D-bcsstk01 & 48 & 176 & 2037.39 & \textbf{2225} & 8.43 & \textit{TL} & 1986.34 & \textbf{2220.05} & \textbf{2225} & 0.22 & \textit{TL} & 2210.31 & 2047.87 & \textbf{2225} & 7.96 & 1 & 2226 & \textit{TL} & 2244 & 0 & 1936 & 0 \\ 
  E-bcspwr02 & 49 & 59 & \textbf{471} & \textbf{471} & 0.00 & 0 & 471.00 & \textbf{471} & \textbf{471} & 0.00 & 0 & 471.00 & 457.87 & \textbf{471} & 2.79 & 0 & \textbf{471} & \textit{TL} & \textbf{471} & 0 & 452 & 0 \\ 
  F-curtis54 & 54 & 124 & 1341 & \textbf{1342} & 0.07 & \textit{TL} & 1295.13 & \textbf{1342} & \textbf{1342} & 0.00 & 15 & 1341.00 & 1292.06 & \textbf{1342} & 3.72 & 1 & 1347 & \textit{TL} & 1347 & 0 & 1233 & 0 \\ 
  G-will57 & 57 & 127 & 1361.25 & \textbf{1369} & 0.57 & \textit{TL} & 1296.83 & \textbf{1369} & \textbf{1369} & 0.00 & 11 & 1369.00 & 1254.20 & \textbf{1369} & 8.39 & 1 & 1375 & \textit{TL} & 1379 & 0 & 1211 & 0 \\ 
  H-impcol-b & 59 & 281 & 3222.25 & \textbf{3363} & 4.19 & \textit{TL} & 3162.00 & \textbf{3349.33} & \textbf{3363} & 0.41 & \textit{TL} & 3346.62 & 3071.46 & \textbf{3363} & 8.67 & 4 & 3387 & \textit{TL} & 3378 & 0 & 3001 & 0 \\ 
  I-ash85 & 85 & 219 & 4114.28 & \textbf{4412} & 6.75 & \textit{TL} & 4071.16 & \textbf{4412} & \textbf{4412} & 0.00 & 185 & 4412.00 & 4093.26 & \textbf{4412} & 7.22 & 5 & 4600 & \textit{TL} & 4444 & 0 & 3890 & 0 \\ 
  jgl009 & 9 & 32 & \textbf{95} & \textbf{95} & 0.00 & 0 & 84.75 & \textbf{95} & \textbf{95} & 0.00 & 1 & 92.17 & 85.86 & \textbf{95} & 9.62 & 0 & \textbf{95} & 52 & \textbf{95} & 0 & 83 & 0 \\ 
  jgl011 & 11 & 49 & \textbf{175} & \textbf{175} & 0.00 & 4 & 150.75 & \textbf{175} & \textbf{175} & 0.00 & 5 & 165.73 & 151.47 & \textbf{175} & 13.45 & 0 & \textbf{175} & \textit{TL} & \textbf{175} & 0 & 147 & 0 \\ 
  J-nos4 & 100 & 247 & 5455 & \textbf{5658} & 3.59 & \textit{TL} & 5416.00 & \textbf{5539.5} & \textbf{5658} & 2.09 & \textit{TL} & 5523.50 & 5285.67 & \textbf{5658} & 6.58 & 3 & 5996 & \textit{TL} & 5667 & 0 & 5218 & 0 \\ 
  K-dwt--234 & 117 & 162 & \textbf{2169} & \textbf{2169} & 0.00 & 3 & 2169.00 & \textbf{2169} & \textbf{2169} & 0.00 & 2 & 2169.00 & 2108.79 & \textbf{2169} & 2.78 & 2 & 2172 & \textit{TL} & \textbf{2169} & 0 & 2105 & 0 \\ 
  L-bcspwr03 & 118 & 179 & \textbf{3557} & \textbf{3557} & 0.00 & 28 & 3546.25 & \textbf{3557} & \textbf{3557} & 0.00 & 13 & 3557.00 & 3440.42 & \textbf{3557} & 3.28 & 3 & 3695 & \textit{TL} & \textbf{3557} & 0 & 3418 & 0 \\ 
  rgg010 & 10 & 45 & \textbf{165} & \textbf{165} & 0.00 & 93 & 135.00 & \textbf{165} & \textbf{165} & 0.00 & 193 & 145.29 & 139.21 & \textbf{165} & 15.63 & 0 & \textbf{165} & \textit{TL} & \textbf{165} & 0 & 135 & 0 \\ 
   \bottomrule
\end{tabular}
\endgroup
\end{table}

\begin{table}[h!tb]
\centering
\caption{Results for \blue{instances of the class} \RND-\sm \label{ta:rndsm}} 
\begingroup\scriptsize
\begin{tabular}{lll|rrrrr|rrrrr|rrrr|rr|rr|rr}
  \toprule
   & & & \multicolumn{5}{c}{\FOIB} & \multicolumn{5}{c}{\FTVB}  & \multicolumn{4}{c}{\LAG} & \multicolumn{2}{c}{\CP} & \multicolumn{2}{c}{\HEUR} & \multicolumn{2}{c}{\DUAL}\\ name & $|V|$ & $|E|$ & $z_D$ & $z^*$ & $g [\%]$ & $t [s]$ & $z_R$ & $z_D$ & $z^*$ & $g [\%]$ & $t [s]$  & $z_R$ & $z_D$ & $z^*$ & $g [\%]$ & $t [s]$ & $z^*$ & $t [s]$ & $z^*$ & $t [s]$ & $z_D$ & $t [s]$ \\ \midrule
random50-100-1 & 50 & 100 & \textbf{911} & \textbf{911} & 0.00 & 1 & 911.00 & \textbf{911} & \textbf{911} & 0.00 & 1 & 911.00 & 871.84 & \textbf{911} & 4.30 & 0 & \textbf{911} & \textit{TL} & \textbf{911} & 0 & 858 & 0 \\ 
  random50-100-2 & 50 & 100 & \textbf{1053} & \textbf{1053} & 0.00 & 4 & 1033.00 & \textbf{1053} & \textbf{1053} & 0.00 & 4 & 1053.00 & 971.08 & \textbf{1053} & 7.78 & 0 & \textbf{1053} & \textit{TL} & 1062 & 0 & 956 & 0 \\ 
  random50-100-3 & 50 & 100 & \textbf{994} & \textbf{994} & 0.00 & 2 & 987.00 & \textbf{994} & \textbf{994} & 0.00 & 3 & 994.00 & 946.17 & \textbf{994} & 4.81 & 0 & 996 & \textit{TL} & 999 & 0 & 923 & 0 \\ 
  random50-100-4 & 50 & 100 & \textbf{1039} & \textbf{1039} & 0.00 & 50 & 1000.00 & \textbf{1039} & \textbf{1039} & 0.00 & 22 & 1023.50 & 962.79 & \textbf{1039} & 7.33 & 0 & \textbf{1039} & \textit{TL} & 1048 & 0 & 933 & 0 \\ 
  random50-100-5 & 50 & 100 & \textbf{978} & \textbf{978} & 0.00 & 7 & 960.00 & \textbf{978} & \textbf{978} & 0.00 & 14 & 972.50 & 909.90 & \textbf{978} & 6.96 & 0 & 985 & \textit{TL} & 990 & 0 & 887 & 0 \\ 
  random50-150-1 & 50 & 150 & \textbf{1599} & \textbf{1599} & 0.00 & 44 & 1543.50 & \textbf{1599} & \textbf{1599} & 0.00 & 9 & 1599.00 & 1487.13 & \textbf{1599} & 7.00 & 1 & 1610 & \textit{TL} & 1611 & 0 & 1447 & 0 \\ 
  random50-150-2 & 50 & 150 & 1674 & \textbf{1736} & 3.57 & \textit{TL} & 1606.00 & \textbf{1736} & \textbf{1736} & 0.00 & 250 & 1711.00 & 1554.26 & \textbf{1736} & 10.47 & 1 & 1739 & \textit{TL} & 1741 & 0 & 1520 & 0 \\ 
  random50-150-3 & 50 & 150 & \textbf{1593} & \textbf{1593} & 0.00 & 12 & 1553.70 & \textbf{1593} & \textbf{1593} & 0.00 & 9 & 1593.00 & 1484.96 & \textbf{1593} & 6.78 & 1 & 1598 & \textit{TL} & \textbf{1593} & 0 & 1469 & 0 \\ 
  random50-150-4 & 50 & 150 & \textbf{1612} & \textbf{1612} & 0.00 & 145 & 1532.50 & \textbf{1612} & \textbf{1612} & 0.00 & 21 & 1608.22 & 1461.63 & \textbf{1612} & 9.33 & 1 & 1623 & \textit{TL} & 1628 & 0 & 1421 & 0 \\ 
  random50-150-5 & 50 & 150 & \textbf{1618} & \textbf{1618} & 0.00 & 188 & 1543.00 & \textbf{1618} & \textbf{1618} & 0.00 & 13 & 1618.00 & 1510.53 & \textbf{1618} & 6.64 & 1 & 1621 & \textit{TL} & 1621 & 0 & 1473 & 0 \\ 
  random50-200-1 & 50 & 200 & 2236.5 & \textbf{2326} & 3.85 & \textit{TL} & 2181.53 & \textbf{2323} & \textbf{2326} & 0.13 & \textit{TL} & 2308.88 & 2131.75 & \textbf{2326} & 8.35 & 1 & 2331 & \textit{TL} & 2340 & 0 & 2059 & 0 \\ 
  random50-200-2 & 50 & 200 & 2314 & \textbf{2470} & 6.32 & \textit{TL} & 2254.00 & \textbf{2438.06} & \textbf{2470} & 1.29 & \textit{TL} & 2424.60 & 2233.40 & \textbf{2470} & 9.58 & 1 & 2484 & \textit{TL} & 2478 & 0 & 2152 & 0 \\ 
  random50-200-3 & 50 & 200 & 2276.5 & \textbf{2374} & 4.11 & \textit{TL} & 2209.50 & \textbf{2365.31} & \textbf{2374} & 0.37 & \textit{TL} & 2344.50 & 2153.92 & \textbf{2374} & 9.27 & 1 & 2386 & \textit{TL} & \textbf{2374} & 0 & 2089 & 0 \\ 
  random50-200-4 & 50 & 200 & 2213 & \textbf{2373} & 6.74 & \textit{TL} & 2128.00 & \textbf{2339.81} & \textbf{2373} & 1.40 & \textit{TL} & 2318.12 & 2078.76 & \textbf{2373} & 12.40 & 1 & 2381 & \textit{TL} & 2381 & 0 & 2000 & 0 \\ 
  random50-200-5 & 50 & 200 & 2244 & \textbf{2324} & 3.44 & \textit{TL} & 2155.50 & \textbf{2324} & \textbf{2324} & 0.00 & 74 & 2317.47 & 2136.98 & \textbf{2324} & 8.05 & 1 & 2332 & \textit{TL} & 2364 & 0 & 2056 & 0 \\ 
   \bottomrule
\end{tabular}
\endgroup
\end{table}

\end{landscape}

Finally, we turn our attention to the larger instances, i.e., \HB-\lag\ and \RND-\lag. For these instances, we tested \LAG, \FOI\ (due to the size of the instances, the LP-solving time becomes prohibitive for any setting based on \second), \HEUR\ and also report on the performance of \DUAL. For dealing with these larger instances with \FOI, instead of the default CPLEX setting for the LP-algorithm and pricing-strategy, we explicitly set the primal simplex algorithm with reduced-cost pricing. This turned out to be beneficial in preliminary runs and is motivated by the fact, that there are many more columns than rows in the formulation and the formulation is very dense. Note that for these larger instances and the given timelimit \FOI\ and \FOIB\ are the same, as the root node of the branch-and-\blue{cut} tree is not finished within the timelimit for any instance. Table \ref{ta:hblag} gives the results for \HB-\lag\ and Table \ref{ta:rndlag} for \RND-\lag. 

We see, that regarding primal solutions, \LAG\ find the best solution value for all instances, while \HEUR\ for none and \FOI\ only for one, so the Lagrangian approach seems to be quite helpful for finding good primal solutions. However, compared to the small instances, \LAG\ now becomes more time consuming and for instances of \HB-\lag\ some runs terminate due to the timelimit. Regarding dual bounds, for instance class \HB-\lag\, \FOI\ finds the best bound for five instances, and \LAG\ for seven. \blue{On} the other hand, for \RND-\lag\, \FOI\ finds the best bound for all instances, while \LAG\ for none. The gaps provided by both \FOI\ and \LAG\ are comparable and between 3\% and 18\%, for \HB-\lag, \LAG\ seems slightly better and for \RND-\lag\ \FOI. The dual bounds provided by \DUAL\ are not far from the bounds provided by \FOI\ and \LAG, while the runtime is much faster (i.e., at most one second). Thus, a branch-and-bound based on a dual heuristic like \DUAL\ could be interesting to explore in further work.

\begin{table}[ht]
\centering
\caption{Results for \blue{instances of the class} \HB-\lag \label{ta:hblag}} 
\begingroup\scriptsize
\begin{tabular}{lll|rrr|rrrr|rr|rr}
  \toprule
   & & & \multicolumn{3}{c}{\FOI} & \multicolumn{4}{c}{\LAG} & \multicolumn{2}{c}{\HEUR} & \multicolumn{2}{c}{\DUAL}\\ name & $|V|$ & $|E|$ & $z_D$ & $z^*$ & $g [\%]$ & $z_D$ & $z^*$ & $g [\%]$ & $t [s]$ & $z^*$ & $t [s]$ & $z_D$ & $t [s]$ \\ \midrule
M-bcsstk06.mtx & 420 & 3720 & 309263.31 & 377439 & 18.06 & \textbf{323231.00} & \textbf{376169} & 14.07 & \textit{TL} & 383166 & 0 & 306337 & 2 \\ 
  N-bcsstk07.mtx & 420 & 3720 & 309263.31 & 377439 & 18.06 & \textbf{323231.00} & \textbf{376169} & 14.07 & \textit{TL} & 383166 & 0 & 306337 & 2 \\ 
  O-impcol-d.mtx & 425 & 1267 & \textbf{97924.49} & 103300 & 5.20 & 94650.20 & \textbf{102501} & 7.66 & 277 & 103312 & 0 & 93511 & 0 \\ 
  P-can--445.mtx & 445 & 1682 & 170217.77 & 197273 & 13.71 & \textbf{178140.00} & \textbf{196762} & 9.46 & \textit{TL} & 198158 & 0 & 169953 & 0 \\ 
  Q-494-bus.mtx & 494 & 586 & \textbf{43972.02} & \textbf{43999} & 0.06 & 42620.30 & \textbf{43999} & 3.13 & 125 & 44418 & 0 & 42477 & 0 \\ 
  R-dwt--503.mtx & 503 & 2762 & 258088.51 & 316834 & 18.54 & \textbf{270426.00} & \textbf{316403} & 14.53 & \textit{TL} & 317347 & 0 & 250309 & 1 \\ 
  S-sherman4.mtx & 546 & 1341 & \textbf{162642.47} & 169284 & 3.92 & 162372.00 & \textbf{168914} & 3.87 & 66 & 171364 & 0 & 162372 & 0 \\ 
  T-dwt--592.mtx & 592 & 2256 & 295127 & 342013 & 13.71 & \textbf{309821.00} & \textbf{341088} & 9.17 & \textit{TL} & 342583 & 0 & 295133 & 1 \\ 
  U-662-bus.mtx & 662 & 906 & \textbf{93056.25} & 95509 & 2.57 & 91209.90 & \textbf{95173} & 4.16 & 347 & 96009 & 0 & 91005 & 0 \\ 
  V-nos6.mtx & 675 & 1290 & 208656 & 211923 & 1.54 & \textbf{208658.00} & \textbf{211908} & 1.53 & 99 & 214026 & 0 & 208658 & 0 \\ 
  W-685-bus.mtx & 685 & 1282 & \textbf{150721.43} & 162327 & 7.15 & 150436.00 & \textbf{161821} & 7.04 & 569 & 162327 & 0 & 148193 & 0 \\ 
  X-can--715.mtx & 715 & 2975 & 409861.59 & 465576 & 11.97 & \textbf{426068.00} & \textbf{464250} & 8.22 & \textit{TL} & 465576 & 0 & 409014 & 1 \\ 
   \bottomrule
\end{tabular}
\endgroup
\end{table}

\begin{table}[h!tb]
\centering
\caption{Results for \blue{instances of the class} \RND-\lag \label{ta:rndlag}} 
\begingroup\scriptsize
\begin{tabular}{lll|rrr|rrrr|rr|rr}
  \toprule
   & & & \multicolumn{3}{c}{\FOI} & \multicolumn{4}{c}{\LAG} & \multicolumn{2}{c}{\HEUR} & \multicolumn{2}{c}{\DUAL}\\ name & $|V|$ & $|E|$ & $z_D$ & $z^*$ & $g [\%]$ & $z_D$ & $z^*$ & $g [\%]$ & $t [s]$ & $z^*$ & $t [s]$ & $z_D$ & $t [s]$ \\ \midrule
random500-1000-1 & 500 & 1000 & \textbf{87139.6} & 92236 & 5.53 & 84429.00 & \textbf{91746} & 7.98 & 47 & 92851 & 0 & 83903 & 0 \\ 
  random500-1000-2 & 500 & 1000 & \textbf{87546.29} & 91361 & 4.18 & 84553.00 & \textbf{91124} & 7.21 & 45 & 91607 & 0 & 84120 & 0 \\ 
  random500-1000-3 & 500 & 1000 & \textbf{86763.58} & 90970 & 4.62 & 83530.50 & \textbf{90636} & 7.84 & 161 & 91605 & 0 & 82662 & 0 \\ 
  random500-1000-4 & 500 & 1000 & \textbf{87855.94} & 91805 & 4.30 & 84837.00 & \textbf{91470} & 7.25 & 47 & 92578 & 0 & 84551 & 0 \\ 
  random500-1000-5 & 500 & 1000 & \textbf{85586.29} & 89547 & 4.42 & 82675.00 & \textbf{89305} & 7.42 & 159 & 90019 & 0 & 82139 & 0 \\ 
  random500-1500-1 & 500 & 1500 & \textbf{140414.03} & 151689 & 7.43 & 136347.00 & \textbf{150407} & 9.35 & 226 & 152341 & 0 & 134706 & 0 \\ 
  random500-1500-2 & 500 & 1500 & \textbf{139672.08} & 152517 & 8.42 & 136926.00 & \textbf{151704} & 9.74 & 237 & 152517 & 0 & 135266 & 0 \\ 
  random500-1500-3 & 500 & 1500 & \textbf{139741.86} & 149669 & 6.63 & 135483.00 & \textbf{149240} & 9.22 & 224 & 150462 & 0 & 133744 & 0 \\ 
  random500-1500-4 & 500 & 1500 & \textbf{140033.31} & 153100 & 8.53 & 136242.00 & \textbf{151839} & 10.27 & 235 & 153315 & 0 & 135021 & 0 \\ 
  random500-1500-5 & 500 & 1500 & \textbf{141017.18} & 154663 & 8.82 & 138273.00 & \textbf{154249} & 10.36 & 65 & 155512 & 0 & 137289 & 0 \\ 
  random500-2000-1 & 500 & 2000 & \textbf{189879.38} & 214580 & 11.51 & 189065.00 & \textbf{213581} & 11.48 & 320 & 214778 & 0 & 187004 & 1 \\ 
  random500-2000-2 & 500 & 2000 & \textbf{193594.01} & 221148 & 12.46 & 192783.00 & \textbf{219807} & 12.29 & 325 & 221148 & 0 & 190568 & 1 \\ 
  random500-2000-3 & 500 & 2000 & \textbf{192637.36} & 214499 & 10.19 & 188478.00 & \textbf{212952} & 11.49 & 332 & 214499 & 0 & 186467 & 1 \\ 
  random500-2000-4 & 500 & 2000 & \textbf{196162.85} & 217207 & 9.69 & 192015.00 & \textbf{215966} & 11.09 & 301 & 217594 & 0 & 189697 & 1 \\ 
  random500-2000-5 & 500 & 2000 & \textbf{198630.17} & 220397 & 9.88 & 194450.00 & \textbf{219567} & 11.44 & 322 & 220397 & 0 & 192649 & 1 \\ 
   \bottomrule
\end{tabular}
\endgroup
\end{table}

\section{Conclusions \label{sec:con}}

In this work, we studied the recently introduced $S$-labeling problem, in which the nodes get labeled using labels from 1 to $|V|$ and for each edge the contribution to the objective function, called $S$-labeling number of the graph, is the minimum label of its end-nodes. The goal is to find a labeling $\phi^*$ with minimum value. We presented two Mixed-Integer Programming (MIP) formulations \blue{(denoted as \first\ and \second)} for the problem and developed branch-and-cut solution frameworks based on \blue{them}. These frameworks were enhanced with valid inequalities, starting and primal heuristics, and specialized branching rules. 
We showed that \blue{formulation \first\ is the projection of formulation \second. Moreover, we also showed that our MIP formulations have no integrality gap for paths, cycles and perfect $n$-ary trees. We proved this with the help of a (heuristic) algorithm to solve the dual of \second.} We gave, to the best of our knowledge, the first polynomial-time algorithm for the problem on $n$-ary trees as well as a closed formula for the $S$-labeling number. Finally, we also presented a Lagrangian heuristic and a constraint programming approach.

We assessed the efficiency of our proposed solution methods in a computational study. The study reveals, that for caterpillar graphs and bipartite graphs our formulation may also have no integrality gap, unfortunately, we were not able to prove any results on this. Moreover, our MIP-approaches are quite effective for solving the \SLP \ to optimality on general graphs with up to around 100 nodes within the timelimit of 600 seconds, and the proposed enhancements, especially the valid inequalities, are helpful. For larger graphs with up to 1000 nodes, the Lagrangian heuristic produces solutions with an optimality gap of about 5-15\% for most of the considered instances and the smaller-sized MIP formulation also provides good results (the size of the MIP formulations becomes burdensome for these larger graphs).
There are various avenues for further work: i) further enhancing the presented MIP-approaches by e.g., additional valid inequalities or other techniques; ii)
development of \blue{approaches to deal with large-scale instances. For the exact solution of large-scale instances, trying to find alternative (smaller-sized) MIP-approaches or developing a (combinatorial) branch-and-bound based on the dual heuristic or the Lagrangian relaxation could be interesting.
Aside from exact algorithms, the design of (meta-)heuristic approaches for large-scale instances could also be a worthwhile topic;} iii) further investigation, if there are additional graph classes, where the problem can be solved in polynomial-time, in particular, bipartite graphs could be an interesting class, as our presented MIP may have no integrality gap.

\section*{Acknowledgements}

The research was supported by the Austrian Research Fund (FWF): P 26755-N19 and P 31366-NBL.


\end{document}